\documentclass[12pt]{amsart}

\usepackage[mathscr]{eucal}
\usepackage{amsmath,amssymb,amscd,amsthm}
\usepackage{color}
\usepackage{epsfig}
\newtheorem{theorem}{Theorem}

\newtheorem{definition}{Definition}
\newtheorem{lemma}{Lemma}
\newtheorem{cor}{Corollary}
\newtheorem{proposition}{Proposition}
\theoremstyle{definition}
\newtheorem{remark}{Remark}

\theoremstyle{plane}

\def \beq{ \begin{equation} }
\def \eeq{\end{equation}}
\def  \q {{\bf q}}
\def \snorm{\sqrt{\q_i\cdot \q_i}\sqrt{\q_j\cdot\q_j}}
\def \hnorm{\sqrt{-\q_i\boxdot \q_i}\sqrt{-\q_j\boxdot\q_j}}

\title{The $n$-body problem in spaces of constant curvature}
\begin{document}
\maketitle
\markboth{F. Diacu, E. P\'erez-Chavela, and M. Santoprete}{The $n$-Body Problem in
Spaces of Constant Curvature}
\author{\begin{center}
Florin Diacu\\
\smallskip
{\footnotesize Pacific Institute for the Mathematical Sciences\\
and\\
Department of Mathematics and Statistics\\
University of Victoria\\
P.O.~Box 3060 STN CSC\\
Victoria, BC, Canada, V8W 3R4\\
diacu@math.uvic.ca\\
}\end{center}

\begin{center}
Ernesto P\'erez-Chavela\\
\smallskip
{\footnotesize
Departamento de Matem\'aticas\\
Universidad Aut\'onoma Metropolitana-Iztapalapa\\
Apdo.\ 55534, M\'exico, D.F., M\'exico\\
epc@xanum.uam.mx\\
}\end{center}

\begin{center}
Manuele Santoprete\\
\smallskip
{\footnotesize Department of Mathematics\\
Wilfrid Laurier University\\
75 University Avenue West,\\
Waterloo, ON, Canada, N2L 3C5.\\
msantopr@wlu.ca\\
}\end{center}
}

\vskip0.5cm

\begin{center}
\today
\end{center}

\bigskip

\begin{abstract}
{We generalize the Newtonian $n$-body problem to spaces of curvature 
$\kappa={\rm constant}$, and study the motion in the 2-dimensional case. For $\kappa>0$, the equations of motion encounter non-collision singularities, which occur when two bodies are antipodal. This phenomenon leads, on one hand, to hybrid solution singularities for as few as 3 bodies, whose corresponding orbits end up in a collision-antipodal configuration in finite time; on the other hand, it produces
non-singularity collisions, characterized by finite velocities and forces
at the collision instant. We also point out the existence of several classes of relative equilibria, including the hyperbolic rotations for $\kappa<0$. In the end, we prove Saari's conjecture when the bodies are on a geodesic that rotates elliptically 
or hyperbolically. We also emphasize that fixed points are specific to the case 
$\kappa>0$, hyperbolic relative equilibria to $\kappa<0$, and Lagrangian orbits of arbitrary masses to $\kappa=0$---results that provide new criteria towards understanding the large-scale geometry of the physical space.}
\end{abstract}

\newpage

\tableofcontents

\section{Introduction}



The goal of this paper is to extend the Newtonian $n$-body problem of celestial
mechanics to spaces of constant curvature. Though attempts of this kind existed 
for two bodies in the 19th century, they faded away after the birth of special and general relativity, to be resurrected several decades later, but only in the case $n=2$. As we will further argue, the topic we are opening here is important for understanding particle dynamics in spaces other than Euclidean, for shedding some new light on the classical case, and perhaps helping us understand the nature of the physical space.


\subsection{History of the problem}

The first researcher who took the idea of gravitation beyond ${\bf R}^3$ was Nikolai Lobachevsky. In 1835, he proposed a Kepler problem in the 3-dimensional hyperbolic space, ${\bf H}^3$, by defining an attractive force proportional to the inverse area of the 2-dimensional sphere of the same radius as the distance between bodies, \cite{Lob}. Independently of him, and at about the same time, J\'anos Bolyai came up with a similar idea, \cite{Bol}. 

These co-discoverers of the first non-Euclidean geometry had no followers in their pre-relativistic attempts until 1860, when Paul Joseph Serret\footnote{Paul Joseph Serret (1827-1898) should not be confused with another French mathematician, Joseph Alfred Serret (1819-1885), known for the Frenet-Serret formulas of vector calculus.} extended the gravitational force to the sphere ${\bf S}^2$ and solved the corresponding Kepler problem, \cite{Ser}. Ten years later, Ernst Schering revisited Lobachevsky's law for which he obtained an analytic expression given by the
cotangent potential we study in this paper, \cite{Sche}. Schering also wrote that Lejeune Dirichlet had told some friends to have dealt with the same problem during his last years in Berlin\footnote{This must have happened around 1852, as
claimed by Rudolph Lipschitz, \cite{Lip72}.}, \cite{Sche1}. In 1873, Rudolph 
Lipschitz considered the same problem in ${\bf S}^3$, but defined a potential proportional to $1/\sin({r/R})$, where $r$ denotes the distance between bodies and $R$ is the curvature radius, \cite{Lip}. He obtained the general solution of this problem in terms of elliptic functions, but his failure to provide an explicit formula invited new approaches. 

In 1885, Wilhelm Killing adapted Lobachevsky's idea to ${\bf S}^3$ and defined an 
extension of the Newtonian force given by the inverse area of a 2-dimensional sphere (in the spirit of Schering), for which he proved a generalization of Kepler's three laws, \cite{Kil10}. Another contributor was Heinrich Liebmann,\footnote{Although he signed his works as Heinrich Liebmann, his full name was Karl Otto Heinrich Liebmann (1874-1939). He did most of his work in Heidelberg and Munich.}. In 1902, he showed that the orbits of the two-body problem are
conics in ${\bf S}^3$ and ${\bf H}^3$ and generalized Kepler's three laws to
$\kappa\ne 0$, \cite{Lie1}. One year later, Liebmann proved ${\bf S}^2$- and ${\bf H}^2$-analogues of Bertrand's theorem, \cite{Ber}, \cite{Win}, which states that there exist only two analytic central potentials in the Euclidean space for which all bounded orbits are closed, \cite{Lie2}. He also summed up his results in a book published in 1905, \cite{Lie3}.

Unfortunately, this direction of research was neglected in the decades following 
the birth of special and general relativity. Starting with 1940, however, Erwin 
Schr\"odinger developed a quantum-mechanical analogue of the Kepler problem 
in ${\bf S}^2$, \cite{Schr}. Schr\"odinger used the same cotangent potential of
Schering and Liebmann, which he deemed to be the natural extension of
Newton's law to the sphere\footnote{``The correct form of [the] potential (corresponding to $1/r$ of the flat space) is known to be $\cot\chi$,'' \cite{Schr}, p.~14.}. Further results in this direction were obtained by 
Leopold Infeld, \cite{Inf}, \cite{Ste}. In 1945, Infeld and his student Alfred Schild extended this problem to spaces of constant negative curvature using a potential given by the hyperbolic cotangent of the distance. A list of the above-mentioned works also appears in \cite{Sh}, except for Serret's book, \cite{Ser}.
A bibliography of works about mechanical problems in spaces
of constant curvature is given in \cite{Shch2}.

Several members of the Russian school of celestial mechanics, including Valeri Kozlov and Alexander Harin, \cite{Koz}, \cite{Koz2}, Alexey Borisov, Ivan Mamaev, and Alexander Kilin, \cite{Bor}, \cite{Bor1}, \cite{Bor2}, \cite{Bor3},  \cite{Kilin}, Alexey Shchepetilov, \cite{Shc}, \cite{Shc1}, \cite{Shch2}, and Tatiana Vozmischeva, \cite{Voz}, revisited the idea of the cotangent potential for the 2-body problem and considered related problems in spaces of constant curvature starting with the 1990s. The main reason for which Kozlov and Harin supported this approach was mathematical. They pointed out, as Schering, Liebmann, Schr\"odinger, Infeld,
and others had insisted earlier, that  (i) the classical one-body problem satisfies Laplace's equation (i.e.~the potential is a harmonic function), which also means that the equations of the problem are equivalent with those of the harmonic oscillator; (ii) its potential generates a central field in which all bounded orbits are closed---according to Bertrand's theorem. Then they showed that the cotangent potential is the only one that satisfies these properties in spaces of constant curvature and has, at the same time, meaning in celestial mechanics. The results they obtained support the idea that the cotangent potential is, so far, the best
extension found for the Newtonian potential to spaces of nonzero constant 
curvature. Our paper brings new arguments that support this view.

The latest contribution to the case $n=2$ belongs to Jos\'e Cari\~nena, Manuel Ra\~nada, and Mariano Santander, who provided a unified approach in the framework of differential geometry, emphasizing the dynamics of the cotangent potential in ${\bf S}^2$ and ${\bf H}^2$, \cite{Car} (see also \cite{Car2}, \cite{Gut}). They also proved that, in this unified context, the conic orbits known in Euclidean space extend naturally to spaces of constant curvature, in agreement with the results obtained by Liebmann, \cite{Sh}.

\subsection{Relativistic $n$-body problems}


Before trying to approach this problem with contemporary tools, we were 
compelled to ask why the direction of research proposed by Lobachevsky was
neglected after the birth of relativity. Perhaps this phenomenon occurred because relativity hoped not only to answer the questions this research direction had asked, but also to regard them from a better perspective than classical mechanics, whose days seemed to be numbered. Things, however, didn't turn out this way. Research on the classical Newtonian $n$-body problem continued and even flourished in the decades to come, and the work on the case $n=2$ in spaces of constant curvature was revived after several decades. But how did relativity fare with respect to this fundamental problem of any gravitational theory?

Although the most important success of relativity was in cosmology and related fields, there were attempts to discretize Einstein's equations and define an $n$-body problem. Reamrkable in this direction were the contributions of Jean Chazy, \cite{Cha}, Tullio Levi-Civita, \cite{Civ}, \cite{Civita}, Arthur Eddington, \cite{Edd}, Albert Einstein, Leopold 
Infeld\footnote{A vivid description of the collaboration between Einstein and Infeld appears in \cite{Infeld}.}, and Banesh Hoffmann, \cite{Ein}, and Vladimir Fock, \cite{Fock}. Subsequent efforts led to refined post-Newtonian approximations (see, e.g., \cite{Soff1}, \cite{Soff2}, \cite{Soff3}), which prove useful in practice, from understanding the motion of artificial satellites---a field with applications in geodesy and geophysics---to using the Global Positioning System (GPS), \cite{Soff4}.

But the equations of the $n$-body problem derived from relativity prove complicated even for $n=2$, and they are not prone to analytical studies similar to the ones done in the classical case. This is probably the reason why the need of some simpler
equations revived the research on the motion of two bodies in spaces of
constant curvature.

Nobody, however, considered the general $n$-body problem\footnote{One of us (Erensto P\'erez-Chavela), together with his student Luis Franco-P\'erez, recently analyzed a restricted 3-body problem in ${\bf S}^1$, \cite{Fra}, in a more restrained context than the one we provide here.} for $n\ge 3$. The lack of developments in this direction may again rest with the complicated form the equations of motion take if one starts from the idea of defining the potential in terms of the intrinsic distance in the framework of differential geometry. Such complications might have discouraged all the attempts to generalize the problem to more than two bodies.

\subsection{Our approach}

The present paper overcomes the above-mentioned difficulties encountered in defining a meaningful $n$-body problem prone to the same mathematical depth achieved in the classical case, by replacing the differential-geometric approach 
used for $n=2$ in the case of the cotangent potential with the variational method of constrained Lagrangian dynamics. Also, the technical complications that arise in understanding the motion within the standard models of the Bolyai-Lobachevsky plane (the Klein-Beltrami disk, the Poincar\'e upper-half-plane, and the Poincar\'e disk) are bypassed through the less known Weierstrass hyperboloidal model (see Appendix), which often provides analogies with the results we obtain in the spherical case. This model also allows us to use hyperbolic rotations---a class of 
isometries---to put into the evidence some unexpected solutions of the equations of motion.


The history of the problem shows that there is no unique way of extending the classical idea of gravitation to spaces of constant curvature, but that the
cotangent potential is the most natural candidate. Therefore we take this potential 
as a starting point of our approach, though some of our results---as for example
Saari's conjecture in the geodesic case---do not use this potential explicitly, only 
its property of being a homogenous function of degree zero. 

Our generalization recovers the Newtonian law when the curvature is zero. Moreover, it provides a unified context, in which the potential varies continuously with the curvature $\kappa$. The same continuity  occurs for the basic results when the curvature tends to zero. For instance, the set of closed orbits of the Kepler problem on non-zero-curvature surfaces tends to the set of ellipses in the Euclidean plane when $\kappa\to 0$ (see, e.g., \cite{Car} or \cite{Lie1}).

\section{Summary of results}

\subsection{Equations of motion}

In Section 3, we extend the Newtonian potential of the $n$-body problem
to spaces of constant curvature, $\kappa$, for any finite dimension. For 
$\kappa\ne 0$, the potential turns out to be a homogeneous function of
degree zero. We also show the existence of an energy integral as well as 
of the integrals of the angular momentum. Like in general relativity, there 
are no integrals of the center of mass and linear momentum. But unlike in 
relativity, where---in the passage from continuous matter to discrete bodies---the 
fact that forces don't cancel at the center of mass leads to difficulties in defining
infinitesimal sizes for finite masses, \cite{Civ}, we do not encounter such 
problems here. We assume that the laws of classical mechanics hold for point
masses moving on manifolds, so we can apply the results of constrained
Lagrangian dynamics to derive the equations of motion. Thus two kinds
of forces act on bodies: (i) those given by the mutual interaction between
particles, represented by the gradient of the potential, and (ii) those that
occur due to the constraints, which involve both position and velocity terms.

\subsection{Singularities}

In Section 4 we focus on singularities, and distinguish between singularities 
of the equations of motion and solution singularities. For any $\kappa\ne 0$,
the equations of motion become singular at collisions, the same as in the
Euclidean case. The case $\kappa>0$, however, introduces some new singularities,
which we call antipodal because they occur when two bodies are at the opposite
ends of a diameter of the sphere.

The set of singularities is endowed with a natural dynamical structure. When the motion of three bodies takes place along a geodesic, solutions close to binary collisions and away from antipodal singularities end up in collision,
so binary collisions are attractive. But antipodal singularities are repulsive
in the sense that no matter how close two bodies are to an antipodal singularity,
they never reach it if the third body is far from a collision with any of them.

Solution singularities arise naturally from the question of existence and uniqueness 
of initial value problems. For nonsingular initial conditions, standard results of the theory of differential equations ensure local existence and uniqueness of an analytic solution defined in some interval $[0,t^+)$. This solution can be analytically extended to an interval $[0,t^*)$, with $0<t^+\le t^*\le\infty$. If $t^*=\infty$, the solution is globally defined. If $t^*<\infty$, the solution is called singular and is said to have a singularity at time $t^*$. 

While the existence of solutions ending in collisions is obvious for any value of $\kappa$, the occurrence of other singularities is not easy to demonstrate. Nevertheless, we prove that some hybrid singular solutions exist in the 3-body problem with $\kappa>0$. These orbits end up in finite time in a collision-antipodal singularity.  Whether other types of non-collision singularities exist, like the pseudocollisions of the Euclidean case, remains an open question. The main reason why this problem is not easy to answer rests with the nonexistence of the center-of-mass integrals.

Another class of solutions connected to collision-antipodal configurations is
particularly interesting. We show that, for $n=3$, there are orbits that reach such a
configuration at some instant $t^*$ but remain analytic at this point because the
forces and the velocities involved remain finite at $t^*$. Such a motion
can, of course, be analytically continued beyond $t^*$. This is the first example
of a natural non-singularity collision.  

\subsection{Relative equilibria}

The rest of this paper, except for the Appendix, focuses on the results we obtained 
in ${\bf S}^2$ and ${\bf H}^2$, mainly because these two surfaces are representative for the cases $\kappa>0$ and $\kappa<0$, respectively. Indeed, the results we proved for these surfaces can be extended to different curvatures of the same sign by a mere change of factor. 

Sections 5 and 6 deal with relative equilibria in ${\bf S}^2$ and ${\bf H}^2$. 
In ${\bf S}^2$ we only have  elliptic relative equilibria. Instead, the relative 
equilibria in ${\bf H}^2$ are of two kinds: elliptic relative equilibria, generated by elliptic rotations, and hyperbolic relative equilibria, generated by hyperbolic rotations (see Appendix). Parabolic relative equilibria, generated by parabolic rotations, do not exist. 

Some of the results we obtain in ${\bf S}^2$ have analogues in ${\bf H}^2$;
others are specific to each case. Theorems \ref{ngonS} and \ref{ngonH}, for 
instance, are dual to each other, whereas Theorem \ref{fix} takes place only
in ${\bf S}^2$.  The latter identifies a class of fixed points of the equations of motion. More precisely, we prove that if an odd number $n$ of equal masses are placed, initially at rest, at the vertices of a regular $n$-gon inscribed in a great circle, then the bodies won't move. The same is true for four equal masses placed at the vertices of a regular tetrahedron inscribed in ${\bf S}^2$, but---due to the occurrence of antipodal singularities---fails to hold for the other regular polyhedra: octahedron (6 bodies), cube (8 bodies), dodecahedron (12 bodies), and icosahedron (20 bodies), as well as in the case of geodesic $n$-gons with an even number of bodies.

Theorem \ref{nofixS} shows that there are no fixed points for $n$ bodies within 
any hemisphere of ${\bf S}^2$. Its hyperbolic analogue, stated in Theorem \ref{nofixH}, proves the nonexistence of fixed points in ${\bf H}^2$. These two results are in agreement with the Euclidean case in the sense that the $n$-body problem has no fixed points within distances, say, not larger than the ray of the visible universe.

It is also natural to ask whether fixed points can generate relative equilibria.
Theorem \ref{fixrel} shows that if $n$ masses $m_1,m_2,
\dots, m_n$ lie initially on a great circle of ${\bf S}^2$ such that the
mutual forces are in equilibrium, then any uniform rotation applied to
the system generates a relative equilibrium.

Theorem \ref{rengon} states that the only way to generate an elliptic relative equilibrium from an initial $n$-gon configuration taken on a great circle, as in Theorem \ref{fix}, is to assign suitable velocities in the plane of the $n$-gon. So 
a regular polygon of this kind can rotate only in a plane orthogonal to the rotation axis. 

Theorem \ref{ngonS} and its hyperbolic analogue, Theorem \ref{ngonH}, show
that $n$-gons of any admissible size can rotate on the same circle, both in ${\bf S}^2$ and ${\bf H}^2$. Again, these results agree with the Euclidean 
case. But something interesting happens with the equilateral (Lagrangian) solutions. Unlike in Euclidean space, elliptic relative equilibria moving in the same plane
of ${\bf R}^3$ can be generated only when the masses move on the same circle and are therefore equal, as we prove in Theorems \ref{lagranS} and \ref{lagranH}. Thus Lagrangian solutions with unequal masses are specific to the Euclidean case.

Theorems \ref{regeo3} and \ref{regeo3H} show that analogues to the collinear (Eulerian) orbits in the 3-body problem of the classical case exist in ${\bf S}^2$ and ${\bf H}^2$, respectively. While nothing surprising happens in ${\bf H}^2$, where
we prove the existence of such solutions of any size, an interesting phenomenon 
takes place in ${\bf S}^2$. Assume that one body lies on the rotation axis (which 
contains one height of the triangle), while the other two are at the opposite ends
of a rotating diameter on some non-geodesic circle of ${\bf S}^2$. Then 
elliptic relative equilibria exist while the bodies are at initial positions 
within the same hemisphere. When the rotating bodies are placed on
the equator, however, they encounter an antipodal singularity. Below the equator, solutions exist again until the bodies form an equilateral triangle. By Theorem \ref{rengon}, any $n$-gon with an odd number of sides can rotate only in its own plane, so the (vertical) equilateral triangle is a fixed point but cannot lead to an elliptic relative equilibrium. If the rotating bodies are then placed below the equilateral position, solutions fail to exist. But the masses don't have to be all
equal. Eulerian solutions exist if, say, the non-rotating body has mass $m$ and
the other two have mass $M$. If $M\ge 4m$, these orbits occur for all
$z\ne 0$. Again, these results prove that, as long as we do not exceed reasonable distances, such as the ray of the visible universe, the behavior of elliptic relative equilibria lying on a rotating geodesic is similar to the one of Eulerian solutions in the Euclidean case.

We then study hyperbolic relative equilibria around a point and along a (usually non-geodesic) hyperbola. Theorem \ref{noreH} proves that such orbits do not exist on fixed geodesics of ${\bf H}^2$, so the bodies cannot chase each other along a geodesic while maintaining the same initial distances. But Theorem \ref{hyp}  proves the existence of hyperbolic relative equilibria in ${\bf H}^2$ for three equal masses. The bodies move along hyperbolas of the hyperboloid that models ${\bf H}^2$ remaining all the time on a moving geodesic and maintaining the initial
distances among themselves. These orbits rather resemble fighter planes flying in 
formation than celestial bodies moving under the action of gravity alone. The
result also holds if the mass in the middle differs from the other two.
The last result of this section, Theorem \ref{thpar}, shows that parabolic relative equilibria do not exist.

\subsection{Saari's conjecture}

Our extension of the Newtonian $n$-body problem to spaces of constant curvature
also reveals new aspects of Saari's conjecture. Proposed in 1970 by Don Saari in
the Euclidean case, Saari's conjecture claims that solutions with constant moment of inertia are relative equilibria. This problem generated a lot of interest from the 
very beginning, but also several failed attempts to prove it. The discovery of the 
figure eight solution, which has an almost constant moment of inertia, and whose existence was proved in 2000 by Alain Chenciner and Richard Montgomery, \cite{Che}, renewed the interest in this conjecture. Several results showed up not long thereafter. The case $n=3$ was solved in 2005 by Rick Moeckel, \cite{Moe}; the collinear case, for any number of bodies and the more general potentials that involve only mutual distances, was settled the same year by the authors of this paper, \cite{Dia2}. Saari's conjecture is also connected to the Chazy-Wintner-Smale conjecture, \cite{Sma}, \cite{Win}, which asks to determine whether the number of central configurations is finite for $n$ given bodies in Euclidean space.

Since relative equilibria have elliptic and hyperbolic versions in ${\bf H}^2$, Saari's conjecture raises new questions for $\kappa<0$. We answered them in Theorem \ref{Saari} of Section 7, when the bodies are restrained to a geodesic that rotates elliptically or hyperbolically.

\bigskip

An Appendix in which we present some basic facts about the Weierstrass model of the hyperbolic plane, together with some historical remarks, closes our paper. We suggest that readers unfamiliar with this model take a look at the Appendix before getting into the technical details related to our results.

\subsection{Some physical remarks}

Does our gravitational model have any connection with the physical reality? Since there is no unique extension of the Newtonian $n$-body problem to spaces of constant curvature, is our generalization meaningful from the physical point of view or does it lead only to some interesting mathematical properties? 

We followed the tradition of the cotangent potential, which seems the most natural candidate. But since the debate on the nature of the physical space is open, the only way to justify this model is through mathematical results. As we will further argue, not only that the properties we obtained match the Euclidean ones, but they also provide a classical explanation of the cosmological scenario, in agreement with the basic conclusions of general relativity. 

But before getting into the physical aspect, let us remark that our model is based on mathematical principles, which lead to a meaningful physical interpretation. As we already mentioned, the cotangent potential preserves two fundamental properties: (i) it is harmonic for the one-body problem and (ii) it generates a central field in which all bounded orbits are closed. Other results that support the  cotangent potential are based on the idea of central (or gnomonic)  projection, \cite{App}. By taking the central projection on the sphere for the planar Kepler problem, Paul Appell obtained the cotangent potential. This idea can be generalized by projecting the planar Kepler problem to any surface of revolution, as one of us
(Manuele Santoprte) proved, \cite{Santoprete}.

In 1992, Kozlov and Harin showed that the only central potential that satisfies the fundamental properties (i) and (ii) in ${\bf S}^2$ and has meaning in celestial mechanics is the cotangent of the distance, \cite{Koz}. This fact had been known to Infeld for the quantum mechanical version of the potential, \cite{Inf}. But since any continuously differentiable and non-constant harmonic function attains no maximum or minimum on the sphere, the existence of two distinct singularities (the collisional and the antipodal---in our case) is not unexpected. And though a force that becomes infinite for points at opposite poles may seem counterintuitive in a gravitational framework, it explains the cosmological scenario. 

Indeed, while there is no doubt that $n$ point masses ejecting from a total collapse
would move forever in  spaces with $\kappa\le0$ for large initial
conditions, in agreement with general relativity, it is not clear 
what happens for $\kappa>0$. But the energy relation \eqref{enerS} shows that, in spherical space, the current expansion of the universe cannot last forever. For a fixed energy constant, $h$, the potential energy, $-U$, would become positive and very large if one or more pairs of particles were to come close to antipodal singularities. Therefore in a homogeneous universe, highly populated with
non-colliding particles, the system could never expand beyond the equator (assuming that the initial ejection took place at one pole). No matter how large (but fixed) the energy constant is, when the potential energy reaches the value $h$, the kinetic energy becomes zero, so all the particles stop simultaneously and the motion reverses.

Thus, for $\kappa>0$, the cotangent potential recovers the growth of the system to a maximum size  and the reversal of the expansion independently on the value of the energy constant. Without antipodal singularities, the reversal could take place only for certain initial conditions. This conclusion is reached without introducing a cosmological force and differently from how it was obtained in the classical model proposed by \'Elie Cartan, \cite{Cart1}, \cite{Cart2}, and shown by Frank Tipler to be as rigorous as Friedmann's cosmology, \cite{Tip1}, \cite{Tip2}.

Another result that suggests the validity of the cotangent potential is the nonexistence of fixed points. They don't show up in the Euclidean case, and neither do they appear in this model within the observable universe. The properties we proved for relative equilibria are also in agreement with the classical $n$-body problem, the only exception being the Lagrangian solutions for $\kappa\ne 0$, which, unlike in the Euclidean case, must have equal masses and move on the same circle. This distinction adds to the strength of the model because, even in the Euclidean case, the arbitrariness of the Lagrangian solutions is a peculiar property.
At least  two arguments support this point of view. 
First, relative equilibria generated from regular polygons, except the equilateral triangle, exist only if the masses are equal. The second argument is related to central configurations, which generate relative equilibria in the Euclidean case. One of us (Florin Diacu) proved that among attraction forces given by symmetric laws of masses, $\gamma(m_i,m_j)=\gamma(m_j,m_i)$,  equilateral central configurations with unequal masses occur only when $\gamma(m_i,m_j)= c\ \! m_im_j$, where $c$ is a nonzero constant, \cite{Diacu}. Since for $\kappa\ne 0$ relative equilibria are equilateral only if the masses are equal means that Lagrangian solutions of arbitrary masses characterize the Euclidean space.

Such orbits exist in nature, the best known example being the equilateral triangle formed by the Sun, Jupiter, and the Trojan asteroids. Therefore our result reinforces the fact that space is Euclidean within distances comparable to those of our solar system. This fact was not known during the time of Gauss, who apparently tried to determine the nature of space by measuring the angles of triangles having the vertices some tens of kilometers apart.\footnote{Arthur Miller argues that these experiments never took place, \cite{Mill}} Since we cannot measure the angles of cosmic triangles, our result opens up a new possibility. Any evidence of a Lagrangian solution involving galaxies (or clusters of galaxies) of unequal masses, could be used as an argument for the flatness of the physical space for distances comparable to the size of the triangle. Similarly, hyperbolic relative equilibria would show that space has negative curvature.


\section{Equations of motion}

We derive in this section a Newtonian $n$-body problem on surfaces of 
constant curvature. The equations of motion we obtain are simple enough 
to allow an analytic approach. At the end, we provide a straightforward 
generalization of these equations to spaces of constant curvature of any finite dimension. 


\subsection{Unified trigonometry} Let us first consider what, following \cite{Car}, we will call trigonometric $\kappa$-functions, which unify elliptical and hyperbolic
trigonometry. We define the 
$\kappa$-sine, ${\rm sn}_\kappa$, as
$$
{\rm sn}_{\kappa}(x):=\left\{
\begin{array}{rl}
{\kappa}^{-1/2}\sin{\kappa}^{1/2}x & {\rm if }\ \ \kappa>0\\
x & {\rm if }\ \ \kappa=0\\
({-\kappa})^{-{1/2}}\sinh({-\kappa})^{1/2}x & {\rm if }\ \ \kappa<0,
\end{array}  \right.
$$
the $\kappa$-cosine, ${\rm csn}_\kappa$, as
$$
{\rm csn}_{\kappa}(x):=\left\{
\begin{array}{rl}
\cos{\kappa}^{1/2}x & {\rm if }\ \ \kappa>0\\
1 & {\rm if }\ \ \kappa=0\\
\cosh({-\kappa})^{1/2}x & {\rm if }\ \ \kappa<0,
\end{array}  \right.
$$
as well as the $\kappa$-tangent, ${\rm tn}_\kappa$, and $\kappa$-cotangent,
${\rm ctn}_\kappa$, as 
$${\rm tn}_{\kappa}(x):={{\rm sn}_{\kappa}(x)\over {\rm csn}_{\kappa}(x)}\ \ \ {\rm and}\ \ \
{\rm ctn}_{\kappa}(x):={{\rm csn}_{\kappa}(x)\over {\rm sn}_{\kappa}(x)},$$
respectively. The entire trigonometry can be rewritten in this unified context,
but the only identity we will further need is the fundamental formula
$$
{\kappa}\ {\rm sn}_{\kappa}^2(x)+{\rm csn}_{\kappa}^2(x)=1.
$$


\subsection{Differential-geometric approach}

In a 2-dimensional Riemann space, we can define geodesic polar
coordinates, $(r,\phi)$, by fixing an origin and an oriented geodesic through it. 
If the space has constant curvature $\kappa$, the range of $r$ depends on $\kappa$; namely $r\in[0,{\pi/(2\kappa^{1/2})}]$ for $\kappa>0$ and $r\in[0,\infty)$ for 
$\kappa\le 0$; in all cases, $\phi\in[0,2\pi]$. The line element is given by
$$ds_{\kappa}^2=dr^2+{\rm sn}_{\kappa}^2(r)d\phi^2.$$
In ${\bf S}^2, {\bf R}^2$, and ${\bf H}^2$, the line element corresponds to $\kappa=1,0,$ and $-1$, respectively, and reduces therefore to 
$$ds_1^2=dr^2+(\sin^2 r)d\phi^2, \ \ \ ds_0^2=dr^2+r^2d\phi^2,\ \  {\rm and}\ \  ds_{-1}^2=dr^2+(\sinh^2 r)d\phi^2.$$

In \cite{Car}, the Lagrangian of the Kepler problem is defined as 
$$L_{\kappa}(r,\phi, v_r, v_{\phi})={1\over 2}[v_r^2+{\rm sn}_{\kappa}^2(r)v_{\phi}^2]+
U_{\kappa}(r),$$
where $v_r$ and $v_{\phi}$ represent the polar components of the velocity, and
$-U$ is the potential, where  
$$U_{\kappa}(r)=G\ {\rm ctn}_{\kappa}(r)$$ 
is the force function, $G>0$ being the gravitational constant. This means that the corresponding force functions in ${\bf S}^2, {\bf R}^2$, and ${\bf H}^2$ are, respectively,
$$U_1(r)={G\cot r}, \ \ \ U_0(r)={G r^{-1}},\ \ \  {\rm and} \ \ \ 
U_{-1}(r)={G\coth r}.$$


In this setting, the case $\kappa=0$ separates the potentials with $\kappa>0$
and $\kappa<0$ into classes exhibiting different qualitative behavior.
The passage from $\kappa>0$ to $\kappa<0$ through $\kappa=0$ takes place continuously. Moreover, the potential is spherically symmetric and satisfies
Gauss's law in a 3-dimensional space of constant curvature $\kappa$. This law
asks that the flux of the radial force field across a sphere of radius $r$ is
a constant independent of $r$. Since the area of the sphere is $4\pi sn_k^2(r)$,
the flux is $4\pi sn_k^2(r)\times{d\over dr}U_{\kappa}(r)$, so the potential 
satisfies Gauss's law. As in the Euclidean case, this generalized potential does
not satisfy Gauss's law in the 2-dimensional space. The results obtained in \cite{Car}
show that the force function $U_{\kappa}$ leads to the expected conic orbits on surfaces of constant curvature, and thus justify this extension of the Kepler problem to $\kappa\ne 0$.


\subsection{The potential}

To generalize the above setting of the Kepler problem to the $n$-body problem
on surfaces of constant curvature, let us start with some notations. Consider
$n$ bodies of masses $m_1,\dots,m_n$ moving on a surface of constant 
curvature $\kappa$. When $\kappa>0$, the surfaces are spheres of radii
${\kappa}^{-1/2}$ given by the equation $x^2+y^2+z^2=\kappa^{-1}$; for 
$\kappa=0$, we recover the Euclidean plane; and if $\kappa<0$, we consider
the Weierstrass model of hyperbolic geometry (see Appendix), which is devised
on the sheets with $z>0$ of the hyperboloids of two sheets 
$x^2+y^2-z^2={\kappa}^{-1}.$
The coordinates of the body of mass $m_i$ are given
by ${\bf q}_i=(x_i,y_i,z_i)$ and a constraint, depending on $\kappa$, 
that restricts the motion of this body to one of the above described surfaces. 

In this paper, ${\widetilde\nabla}_{{\bf q}_i}$ denotes either of the gradient operators
$$
\nabla_{{\bf q}_i}=(\partial_{x_i},\partial_{y_i},\partial_{z_i}),\  {\rm for}\ \ \kappa\ge0,\ \ {\rm or}\ \
{\overline\nabla}_{{\bf q}_i}=(\partial_{x_i},\partial_{y_i},-\partial_{z_i}),\  {\rm for}\ \ \kappa<0,
$$
with respect to the vector ${\bf q}_i$, and $\widetilde\nabla$ stands for the operator
$(\widetilde\nabla_{{\bf q}_1},\dots,\widetilde\nabla_{{\bf q}_n})$.
For ${\bf a}=(a_x,a_y,a_z)$ and ${\bf b}=(b_x,b_y,b_z)$ in ${\bf R}^3$, 
we define ${\bf a}\odot{\bf b}$ as either of the inner products
$${\bf a}\cdot{\bf b}:=(a_xb_x+a_yb_y+a_zb_z) \ \ {\rm for}\ \ \kappa\ge0,$$
$${\bf a}\boxdot{\bf b}:=(a_xb_x+a_yb_y-a_zb_z) \ \ {\rm for}\ \ \kappa<0,$$
the latter being the Lorentz inner product (see Appendix).
We also define ${\bf a}\otimes{\bf b}$ as either of the cross products
$${\bf a}\times{\bf b}:=(a_yb_z-a_zb_y, a_zb_x-a_xb_z, a_xb_y-a_yb_x) \ \ {\rm for}\ \ \kappa\ge0,$$
$${\bf a}\boxtimes{\bf b}:=(a_yb_z-a_zb_y, a_zb_x-a_xb_z, a_yb_x-a_xb_y)\ \ {\rm for}\ \ \kappa<0.$$

The distance between $\bf a$ and $\bf b$ on the surface of constant curvature $\kappa$ 
is then given by
$$
d_{\kappa}({\bf a},{\bf b}):=
\begin{cases}
\kappa^{-1/2}\cos^{-1}(\kappa{\bf a}\cdot{\bf b}),\ \ \ \ \ \ \ \ \  \kappa >0\cr
|{\bf a}-{\bf b}|, \ \ \ \ \ \ \ \ \ \ \ \ \ \ \ \ \ \ \ \ \ \ \  \! \ \hspace{2 pt} \kappa=0\cr
({-\kappa})^{-1/2}\cosh^{-1}(\kappa{\bf a}\boxdot{\bf b}),\hspace{4 pt} \kappa<0,\cr
\end{cases}
$$
where the vertical bars denote the standard Euclidean norm. In particular, the distances in ${\bf S}^2$ and ${\bf H}^2$ are 
$$
d_1({\bf a},{\bf b})=\cos^{-1}({\bf a}\cdot{\bf b}),\ \ 
d_{-1}({\bf a},{\bf b})=\cosh^{-1}(-{\bf a}\boxdot{\bf b}),
$$
respectively. Notice that $d_0$ is the limiting case of $d_\kappa$ when $\kappa\to 0$. Indeed, 
for both $\kappa>0$ and $\kappa<0$, the vectors $\bf a$ and $\bf b$ tend to infinity 
and become parallel, while the surfaces tend to an Euclidean plane, therefore the 
length of the arc between the vectors tends to the Euclidean distance.

We will further define a potential in ${\bf R}^3$ if $\kappa>0$, and in the $3$-dimensional Minkowski space $\mathcal{M}^3$ (see Appendix) if $\kappa<0$,
such that we can use a variational method to derive the equations of motion. For 
this purpose we need to extend the distance to these spaces. We do this by 
redefining the distance as
$$
d_{\kappa}({\bf a},{\bf b}):=
\begin{cases}
\kappa^{-1/2}\cos^{-1}{\kappa{\bf a}\cdot{\bf b}\over\sqrt{\kappa{\bf a}\cdot{\bf a}}
\sqrt{\kappa{\bf b}\cdot{\bf b}}},\ \ \ \ \ \ \ \ \! \ \ \kappa >0\cr
|{\bf a}-{\bf b}|, \ \ \ \ \ \ \ \ \ \ \ \ \ \ \ \ \ \ \ \ \ \ \! \ \hspace{0.82cm} \kappa=0\cr
({-\kappa})^{-1/2}\cosh^{-1}{\kappa{\bf a}\boxdot{\bf b}\over\sqrt{\kappa{\bf a}\boxdot{\bf a}}\sqrt{\kappa{\bf b}\boxdot{\bf b}}},\hspace{4.5 pt} \kappa<0.\cr
\end{cases}
$$
Notice that this new definition is identical with the previous one when we
restrict the vectors ${\bf a}$ and ${\bf b}$ to the spheres $x^2+y^2+z^2=\kappa^{-1}$
or the hyperboloids $x^2+y^2-z^2=\kappa^{-1}$, but is also valid for any
vectors $\bf a$ and $\bf b$  in ${\bf R}^3$
and $\mathcal{M}^3$, respectively.
 
From now on we will rescale the units such that the gravitational constant $G$
is $1$. We thus define the potential of the $n$-body problem as the function 
$-U_\kappa({\bf q})$, where
\begin{equation*}
U_\kappa({\bf q}):={1\over 2}\sum_{i=1}^n\sum_{j=1,j\ne i}^n{m_im_j{\rm ctn}_\kappa
({d}_\kappa({\bf q}_i,{\bf q}_j))}
\label{defpot}
\end{equation*}
stands for the force function, and ${\bf q}=({\bf q}_1,\dots, {\bf q}_n)$ is the configuration of the system.
Notice that ${\rm ctn}_0({d}_0({\bf q}_i,{\bf q}_j))=|{\bf q}_i-{\bf q}_j|^{-1}$, which means that we recover the Newtonian potential in the Euclidean case. 
Therefore the potential $U_\kappa$ varies continuously with the curvature $\kappa$. 

Now that we defined a potential that satisfies the basic continuity condition we 
required of any extension of the $n$-body problem beyond the Euclidean space, 
we will focus on the case $\kappa\ne0$. A straightforward computation shows that
\begin{equation}
U_\kappa({\bf q})={1\over 2}\sum_{i=1}^n\sum_{j=1,j\ne i}^n{m_im_j
(\sigma\kappa)^{1/2}{\kappa{\bf q}_i\odot{\bf q}_j\over\sqrt{\kappa{\bf q}_i
\odot{\bf q}_i}\sqrt{\kappa{\bf q}_j\odot{\bf q}_j}}\over
\sqrt{\sigma-\sigma\Big({\kappa{\bf q}_i\odot{\bf q}_j\over
\sqrt{\kappa{\bf q}_i
\odot{\bf q}_i}\sqrt{\kappa{\bf q}_j\odot{\bf q}_j}}\Big)^2}}, \ \ \kappa\ne 0,
\label{pothom}
\end{equation}
where 
$$
\sigma=
\begin{cases}
+1, \ \ {\rm for} \ \ \kappa>0\cr
-1, \ \ {\rm for} \ \ \kappa<0.\cr
\end{cases}
$$


\subsection{Euler's formula}

Notice that $U_{\kappa}(\eta{\bf q})=U_\kappa({\bf q})=\eta^0U_{\kappa}({\bf q})$
for any $\eta\ne 0$, which means that the potential is a homogeneous function of degree zero. But for ${\bf q}$ 
in ${\bf R}^{3n}$, homogeneous functions $F:{\bf R}^{3n}\to{\bf R}$ of degree $\alpha$ satisfy Euler's formula, ${\bf q}\cdot\nabla F({\bf q})=\alpha F({\bf q})$. With our notations, Euler's formula can be written as
${\bf q}\odot\widetilde\nabla{F({\bf q})}=\alpha F({\bf q})$.
Since $\alpha=0$ for $U_{\kappa}$ with $\kappa\ne 0$, we conclude that
\begin{equation}
{\bf q}\odot\widetilde\nabla U_{\kappa}({\bf q})=0.
\end{equation}
We can also write the force function as 
$U_\kappa({\bf q})={1\over 2}\sum_{i=1}^nU_\kappa^i({\bf q}_i)$, where
$$
U_\kappa^i({\bf q}_i):=
\sum_{j=1,j\ne i}^n{m_im_j(\sigma\kappa)^{1/2}{\kappa{\bf q}_i\odot{\bf q}_j\over\sqrt{\kappa{\bf q}_i\odot{\bf q}_i}\sqrt{\kappa{\bf q}_j\odot{\bf q}_j}}\over\sqrt{\sigma-\sigma\Big({\kappa{\bf q}_i\odot{\bf q}_j\over\sqrt{\kappa{\bf q}_i\odot{\bf q}_i}\sqrt{\kappa{\bf q}_j\odot{\bf q}_j}}\Big)^2}}, \ \ i=1,\dots,n,
$$ 
are also homogeneous functions of degree $0$. Applying Euler's formula for functions $F:{\bf R}^3\to{\bf R}$, we obtain that ${\bf q}_i\odot\widetilde\nabla_{{\bf q}_i} U_\kappa^i({\bf q})=0$. Then using the identity $\widetilde\nabla_{{\bf q}_i}U_\kappa({\bf q})=\widetilde\nabla_{{\bf q}_i} U_\kappa^i({\bf q}_i)$, we can conclude that 
\begin{equation}
{\bf q}_i\odot\widetilde\nabla_{{\bf q}_i} U_\kappa({\bf q})=0, \ \ i=1,\dots,n.
\label{eul}
\end{equation}


\subsection{Derivation of the equations of motion}

To obtain the equations of motion for $\kappa\ne 0$, we will use a variational
method applied to the force function (\ref{pothom}). The Lagrangian of the $n$-body system has the form
$$
L_\kappa({\bf q}, \dot{\bf q})=T_\kappa({\bf q}, \dot{\bf q})+U_\kappa({\bf q}),
$$
where $T_\kappa({\bf q},\dot{\bf q}):={1\over 2}\sum_{i=1}^nm_i(\dot{\bf q}_i\odot\dot{\bf q}_i)(\kappa{\bf q}_i\odot{\bf q}_i)$ is the kinetic energy of the system. 
(The reason for introducing the factors $\kappa{\bf q}_i\odot{\bf q}_i=1$ into the definition of the kinetic energy  will become clear in Section 3.8.)
Then, according to the theory of constrained Lagrangian dynamics  (see, e.g., \cite{Gel}), the equations of motion are
\begin{equation}
{d\over dt}\Bigg({\partial L_\kappa\over\partial\dot{\bf q}_i}\Bigg)-{\partial L_\kappa\over\partial{\bf q}_i}-\lambda_\kappa^i(t){\partial f_i\over\partial{\bf q}_i}={\bf 0},\ \ \ i=1,\dots,n,\label{eqLagrangianS2}
\end{equation}
where $f_\kappa^i={\bf q}_i\odot{\bf q}_i-{\kappa}^{-1}$ is the function that gives the constraint $f_\kappa^i=0$, which keeps the body of mass $m_i$ on the surface of constant curvature $\kappa$, and $\lambda_\kappa^i$ is the Lagrange multiplier corresponding to the same body. Since ${\bf q}_i\odot{\bf q}_i=\kappa^{-1}$ implies that 
$\dot{\bf q}_i\odot{\bf q}_i=0$, it follows that
$$
{d\over dt}\Bigg({\partial L_\kappa\over\partial\dot{\bf q}_i}\Bigg)=
m_i\ddot{\bf q}_i(\kappa{\bf q}_i\odot{\bf q}_i)+2 m_i(\kappa\dot{\bf q}_i\odot{\bf q}_i)=
m_i\ddot{\bf q}_i.
$$
This relation, together with
$$  {\partial L_\kappa\over\partial{\bf q}_i}=m_i\kappa(\dot{\bf q}_i\odot\dot{\bf q}_i){\bf q}_i+\widetilde\nabla_{{\bf q}_i} U_{\kappa}({\bf q}),
$$
implies that equations (\ref{eqLagrangianS2}) are equivalent to
\begin{equation}
m_i\ddot{\bf q}_i-m_i\kappa(\dot{\bf q}_i\odot\dot{\bf q}_i){\bf q}_i-\widetilde\nabla_{{\bf q}_i} U_{\kappa}({\bf q})-2\lambda_\kappa^i(t){\bf q}_i={\bf 0},\ \ \ i=1,\dots, n.\label{equations}
\end{equation}
To determine $\lambda_\kappa^i$, notice that
$0=\ddot{f}_\kappa^i=2\dot{\bf q}_i\odot\dot{\bf q}_i+
2({\bf q}_i\odot\ddot{\bf q}_i),$ so
\begin{equation}
{\bf q}_i\odot\ddot{\bf q}_i=-\dot{\bf q}_i\odot\dot{\bf q}_i.\label{h1}
\end{equation}
Let us also remark that $\odot$-multiplying equations (\ref{equations}) by 
${\bf q}_i$ and using (\ref{eul}), we obtain that
$$
m_i({\bf q}_i\odot\ddot{\bf q}_i)-m_i\kappa(\dot{\bf q}_i\odot\dot{\bf q}_i)-{\bf q}_i\odot\widetilde\nabla_{{\bf q}_i} U_{\kappa}({\bf q})=2\lambda_\kappa^i{\bf q}_i\odot{\bf q}_i=2\kappa^{-1}\lambda_\kappa^i,
$$
which, via (\ref{h1}), implies that $\lambda_\kappa^i=-\kappa m_i(\dot{\bf q}_i\odot\dot{\bf q}_i)$.
Substituting these values of the Lagrange multipliers into equations (\ref{equations}),
the equations of motion and their constraints become
\begin{multline}
m_i \ddot {\bf q}_i=\widetilde\nabla_{{\bf q}_i} U_{\kappa}({\bf q})-m_i\kappa(\dot{\bf q}_i\odot\dot{\bf q}_i){\bf q}_i, \ \ {\bf q}_i\odot{\bf q}_i=\kappa^{-1}, \ \ \kappa\ne 0,\\
 \ \ i=1,\dots, n.\label{eqmotion}
\end{multline}
The ${\bf q}_i$-gradient of the force function, obtained from (\ref{pothom}), has the
form
\begin{equation}
\widetilde\nabla_{{\bf q}_i} U_\kappa({\bf q})=\sum_{j=1,j\ne i}^n
{{m_im_j(\sigma\kappa)^{1/2}\left(\sigma\kappa{\bf q}_j  -\sigma{\kappa^2{\bf q}_i\odot{\bf q}_j\over\kappa{\bf q}_i\odot{\bf q}_i}{\bf q}_i \right)\over\sqrt{\kappa{\bf q}_i\odot{\bf q}_i}\sqrt{\kappa{\bf q}_j\odot{\bf q}_j}}
\over
\left[\sigma-\sigma\left({\kappa{\bf q}_i\odot{\bf q}_j}\over{\sqrt{\kappa{\bf q}_i\odot{\bf q}_i}\sqrt{\kappa{\bf q}_j\odot{\bf q}_j}}\right)^2\right]^{3/2}},\ \kappa\ne 0,
\label{gra}
\end{equation}
and using the fact that $\kappa{\bf q}_i\odot{\bf q}_i=1$, we can write this gradient as
\begin{equation}
\widetilde\nabla_{{\bf q}_i}U_\kappa({\bf q})=\sum_{j=1,j\ne i}^n{{m_im_j}(\sigma
\kappa)^{3/2}
\left[{\bf q}_j  - (\kappa{\bf q}_i\odot{\bf q}_j){\bf q}_i \right]
\over
\left[\sigma-\sigma\left(\kappa{\bf q}_i\odot{\bf q}_j\right)^2\right]^{3/2}}, \ \kappa\ne 0.
\label{grad}
\end{equation}
Sometimes we can use the simpler form \eqref{grad} of the gradient, but whenever
we need to exploit the homogeneity of the gradient or have to differentiate it, we must
use its original form \eqref{gra}. 
Thus equations (\ref{eqmotion}) and (\ref{gra}) describe
the $n$-body problem on surfaces of constant curvature for $\kappa\ne0$. 
Though more complicated than the equations of motion Newton derived for 
the Euclidean space, system (\ref{eqmotion}) is simple enough to allow an
analytic approach. Let us first provide some of its basic properties.


\subsection{First integrals}


The equations of motion have the energy integral
\begin{equation}
T_\kappa({\bf q},{\bf p})-U_\kappa({\bf q})=h,\label{energy}
\end{equation}
where, recall, $T_\kappa({\bf q}, {\bf p}):={1\over 2}\sum_{i=1}^nm_i^{-1}({\bf p}_i\odot{\bf p}_i)(\kappa{\bf q}_i\odot{\bf q}_i)$ is the kinetic energy,  ${\bf p}:=({\bf p}_1,\dots,{\bf p}_n)$ denotes the momentum of the $n$-body system, with ${\bf p}_i:=m_i\dot{\bf q}_i$ representing the momentum of the body of mass $m_i, i=1,\dots,n$, and $h$ is a real constant. Indeed, $\odot$-multiplying equations (\ref{eqmotion}) by $\dot{\bf q}_i$, we obtain 
$$
\sum_{i=1}^nm_i\ddot{\bf q}_i\odot\dot{\bf q}_i=
[\widetilde\nabla_{{\bf q}_i} U_\kappa({\bf q})]\odot\dot{\bf q}_i-
\sum_{i=1}^n{m_i\kappa}(\dot{\bf q}_i\odot\dot{\bf q}_i){\bf q}_i\odot{\dot{\bf q}_i}=
{d\over dt}U_\kappa({\bf q}(t)).
$$
Then equation (\ref{energy}) follows by integrating the first and last term in the above 
equation.


The equations of motion also have the integrals of the angular momentum,
\begin{equation}
\sum_{i=1}^n{\bf q}_i\otimes{\bf p}_i={\bf c},\label{angular} 
\end{equation}
where $\bf c$ is a constant vector. Relations (\ref{angular}) follow by integrating the identity formed by the first
and last term of the equations
\begin{multline}
\sum_{i=1}^nm_i\ddot{\bf q}_i\otimes{\bf q}_i=\sum_{i=1}^n\sum_{j=1,j\ne i}^n{m_im_j(\sigma\kappa)^{3/2}{\bf q}_i\otimes{\bf q}_j
\over
[\sigma-\sigma(\kappa{\bf q}_i\odot{\bf q}_j)^2]^{3/2}}\\
-\sum_{i=1}^n\left[\sum_{j=1,j\ne i}^n{m_im_j(\sigma\kappa)^{3/2}(\kappa{\bf q}_i\odot{\bf q}_j)
\over
[\sigma-\sigma(\kappa{\bf q}_i\odot{\bf q}_j)^2]^{3/2}}-
m_i{\kappa}(\dot{\bf q}_i\odot\dot{\bf q}_i)\right]{\bf q}_i\otimes{\bf q}_i
={\bf 0},
\end{multline}
obtained if $\otimes$-multiplying the equations of motion (\ref{eqmotion}) 
by ${\bf q}_i$. The last of the above identities follows from the skew-symmetry of 
$\otimes$ and the fact that  ${\bf q}_i\otimes{\bf q}_i={\bf 0}, \ i=1,\dots,n$.


\subsection{Motion of a free body}

A consequence of the integrals of motion is the analogue of the well known result from the Euclidean space related to the motion of a single body in the absence of any gravitational interactions. Though simple, the proof of this property is not as trivial as in the classical case.

\begin{proposition}
A free body on a surface of constant curvature is either at rest or it moves uniformly along a geodesic. Moreover, for $\kappa>0$, every orbit is closed.
\end{proposition}
\begin{proof}
Since there are no gravitational interactions, the equations of motion
take the form
\begin{equation}
\ddot{\bf q}=-\kappa(\dot{\bf q}\odot\dot{\bf q}){\bf q},
\end{equation}
where ${\bf q}=(x,y,z)$ is the vector describing the position of the body of mass
$m$. If $\dot{\bf q}(0)={\bf 0}$, then $\ddot{\bf q}(0)={\bf 0}$, so no force acts on 
$m$. Therefore the body will be at rest.

If $\dot{\bf q}(0)\ne{\bf 0}$, $\ddot{\bf q}(0)$ and ${\bf q}(0)$ are collinear, having
the same sense if $\kappa<0$, but the opposite sense if $\kappa>0$. So the 
sum between $\ddot{\bf q}(0)$ and $\dot{\bf q}(0)$ pulls the body 
along the geodesic corresponding to the direction of these vectors. 

We still need to show that the motion is uniform. This fact follows obviously
from the integral of energy, but we can also derive it from the integrals of the
angular momentum. Indeed, for $\kappa>0$, these integrals lead us to
$$c=({\bf q}\times\dot{\bf q})\cdot({\bf q}\times\dot{\bf q})=({\bf q}\cdot{\bf q})
(\dot{\bf q}\cdot\dot{\bf q})\sin^2\alpha,$$
where $c$ is the length of the angular momentum vector and $\alpha$ is the
angle between ${\bf q}$ and $\dot{\bf q}$ (namely $\pi/2$). So since
${\bf q}\cdot{\bf q}=\kappa^{-1}$, we can
draw the conclusion that the speed of the body is constant.

For $\kappa<0$, we can write that
$$
c=({\bf q}\boxtimes\dot{\bf q})\boxdot ({\bf q}\boxtimes\dot{\bf q})=
-\left|\begin{array}{cc}
{\bf q}\boxdot{\bf q} & {\bf q}\boxdot\dot{\bf q}\\
{\bf q}\boxdot\dot{\bf q} & \dot{\bf q}\boxdot\dot{\bf q}\\
\end{array}\right|=-\left|\begin{array}{cc}
\kappa^{-1} & 0\\
0 & \dot{\bf q}\boxdot\dot{\bf q}\\
\end{array}\right|=-\kappa^{-1}\dot{\bf q}\boxdot\dot{\bf q}.
$$
Therefore the speed is constant in this case too, so the motion is uniform. Since
for $\kappa>0$ the body moves on geodesics of a sphere, every orbit is closed.
\end{proof}


\subsection{Hamiltonian form}

The equations of motion (\ref{eqmotion}) are Hamiltonian. Indeed, the Hamiltonian function
$H_\kappa$ is given by
$$
\begin{cases}
H_\kappa({\bf q},{\bf p})=
{1\over 2}\sum_{i=1}^nm_i^{-1}({\bf p}_i\odot{\bf p}_i)
(\kappa{\bf q}_i\odot{\bf q}_i)-U_\kappa({\bf q}),\cr 
{\bf q}_i\odot{\bf q}_i={\kappa}^{-1}, \ \kappa\ne 0,
\ \ 
i=1,\dots,n.
\end{cases}
$$
Equations (\ref{equations}) thus take the form of a $6n$-dimensional first order system
of differential equations with $2n$ constraints,
\begin{equation}
\begin{cases}
\dot{\bf q}_i=
\widetilde\nabla_{{\bf p}_i} H_\kappa({\bf q},{\bf p})=m_i^{-1}{\bf p}_i,\cr
\dot{\bf p}_i=
-\widetilde\nabla_{{\bf q}_i} H_\kappa({\bf q},{\bf p})=
\widetilde\nabla_{{\bf q}_i} U_\kappa({\bf q})
-m_i^{-1}{\kappa}({\bf p}_i\odot{\bf p}_i){\bf q}_i,\cr
{\bf q}_i\odot{\bf q}_i={\kappa}^{-1}, 
\ \  {\bf q}_i\odot{\bf p}_i=0,
\ \ \kappa\ne 0,
\ \  i=1,\dots,n.\label{Ham}
\end{cases}
\end{equation}

It is interesting to note that, independently of whether the kinetic energy is
defined as 
$$T_\kappa({\bf p}):={1\over 2}\sum_{i=1}^nm_i^{-1}{\bf p}_i\odot{\bf p}_i\ \ {\rm or} \ \
T_\kappa({\bf q}, {\bf p}):={1\over 2}\sum_{i=1}^nm_i^{-1}({\bf p}_i\odot{\bf p}_i)(\kappa{\bf q}_i\odot{\bf q}_i),$$
(which, though identical since $\kappa{\bf q}_i\odot{\bf q}_i=1$, does not come to the same thing when differentiating $T_\kappa$), the form
of equations (\ref{eqmotion}) remains the same. But in the former case, system
(\ref{eqmotion}) cannot be put in Hamiltonian form in spite of having an energy
integral, while in the former case it can. This is why we chose the 
latter definition of $T_\kappa$.

These equations describe the motion of the $n$-body system for any $\kappa\ne 0$,
the case $\kappa=0$ corresponding to the classical Newtonian equations. 
The representative non-zero-curvature cases, however, are $\kappa=1$ and $\kappa=-1$, which characterize the motion for $\kappa>0$ and $\kappa<0$, respectively. Therefore we will further focus on the $n$-body problem in ${\bf S}^2$ 
and ${\bf H}^2$.


\subsection{Equations of motion in ${\bf S}^2$}

In this case, the force function (\ref{pothom}) takes the form
\begin{equation}
U_1({\bf q})={1\over 2}\sum_{i=1}^n\sum_{j=1,j\ne i}^n\frac{m_im_j~ \frac{{\bf q}_i\cdot{\bf q}_j}{\snorm}}{
\sqrt{1-\left(\frac{{\bf q}_i\cdot{\bf q}_j}{\snorm}\right)^2}}
,\label{potS}
\end{equation}
while the equations of motion (\ref{eqmotion}) and their constraints become 
\begin{equation}
m_i\ddot{\bf q}_i=\nabla_{{\bf q}_i} U_1({\bf q})-m_i(\dot{\bf q}_i\cdot\dot{\bf q}_i){\bf q}_i,\ \ \ {\bf q}_i\cdot{\bf q}_i=1,\ \ \ {\bf q}_i\cdot\dot{\bf q}_i=0,\ \ \ i=1,\dots,n.\label{eqS}
\end{equation}

In terms of coordinates, the equations of motion and their constraints can be written as
\begin{equation}
\begin{cases}
m_i\ddot{x}_i={\partial U_1\over\partial x_i}-m_i(\dot{x}_i^2+\dot{y}_i^2+\dot{z}_i^2)x_i,\cr
m_i\ddot{y}_i={\partial U_1\over\partial y_i}-m_i(\dot{x}_i^2+\dot{y}_i^2+\dot{z}_i^2)y_i,\cr
m_i\ddot{z}_i={\partial U_1\over\partial z_i}-m_i(\dot{x}_i^2+\dot{y}_i^2+\dot{z}_i^2)z_i,\cr
x_i^2+y_i^2+z_i^2=1,\ \ x_i\dot{x}_i+y_i\dot{y}_i+z_i\dot{z}_i=0,\ \ i=1,\dots,n,\label{eqcoordS}
\end{cases}
\end{equation}
and by computing the gradients they become
\begin{equation}
\begin{cases}
\ddot{x}_i=\sum_{j=1,j\ne i}^n{m_j{x_j-{x_ix_j+y_iy_j+z_iz_j\over x_i^2+y_i^2+z_i^2}x_i
\over \sqrt{x_i^2+y_i^2+z_i^2}\sqrt{x_j^2+y_j^2+z_j^2}}\over
\bigg[1-\Big({x_ix_j+y_iy_j+z_iz_j\over\sqrt{x_i^2+y_i^2+z_i^2}\sqrt{x_j^2+y_j^2+z_j^2}}\Big)^2\bigg]^{3/2}}-(\dot{x}_i^2+\dot{y}_i^2+\dot{z}_i^2)x_i,\cr
\ddot{y}_i=\sum_{j=1,j\ne i}^n{m_j{y_j-{x_ix_j+y_iy_j+z_iz_j\over x_i^2+y_i^2+z_i^2}y_i
\over \sqrt{x_i^2+y_i^2+z_i^2}\sqrt{x_j^2+y_j^2+z_j^2}}\over
\bigg[1-\Big({x_ix_j+y_iy_j+z_iz_j\over\sqrt{x_i^2+y_i^2+z_i^2}\sqrt{x_j^2+y_j^2+z_j^2}}\Big)^2\bigg]^{3/2}}-(\dot{x}_i^2+\dot{y}_i^2+\dot{z}_i^2)y_i,\cr
\ddot{z}_i=\sum_{j=1,j\ne i}^n{m_j{z_j-{x_ix_j+y_iy_j+z_iz_j\over x_i^2+y_i^2+z_i^2}z_i
\over \sqrt{x_i^2+y_i^2+z_i^2}\sqrt{x_j^2+y_j^2+z_j^2}}\over
\bigg[1-\Big({x_ix_j+y_iy_j+z_iz_j\over\sqrt{x_i^2+y_i^2+z_i^2}\sqrt{x_j^2+y_j^2+z_j^2}}\Big)^2\bigg]^{3/2}}-(\dot{x}_i^2+\dot{y}_i^2+\dot{z}_i^2)z_i,\cr
x_i^2+y_i^2+z_i^2=1,\ \ x_i\dot{x}_i+y_i\dot{y}_i+z_i\dot{z}_i=0,\ \ i=1,\dots,n.\label{coordSfull}
\end{cases}
\end{equation}
Since we will neither need the homogeneity of the gradient,
nor we will we differentiate it, we can use the constraints to write the
above system as
\begin{equation}
\begin{cases}
\ddot{x}_i=\sum_{j=1,j\ne i}^n{m_j[x_j-(x_ix_j+y_iy_j+z_iz_j)x_i]\over
[1-(x_ix_j+y_iy_j+z_iz_j)^2]^{3/2}}-(\dot{x}_i^2+\dot{y}_i^2+\dot{z}_i^2)x_i,\cr
\ddot{y}_i=\sum_{j=1,j\ne i}^n{m_j[y_j-(x_ix_j+y_iy_j+z_iz_j)y_i]\over
[1-(x_ix_j+y_iy_j+z_iz_j)^2]^{3/2}}-(\dot{x}_i^2+\dot{y}_i^2+\dot{z}_i^2)y_i,\cr
\ddot{z}_i=\sum_{j=1,j\ne i}^n{m_j[z_j-(x_ix_j+y_iy_j+z_iz_j)z_i]\over
[1-(x_ix_j+y_iy_j+z_iz_j)^2]^{3/2}}-(\dot{x}_i^2+\dot{y}_i^2+\dot{z}_i^2)z_i,\cr
x_i^2+y_i^2+z_i^2=1,\ \ x_i\dot{x}_i+y_i\dot{y}_i+z_i\dot{z}_i=0,\ \ i=1,\dots,n.\label{coordS}
\end{cases}
\end{equation}
The Hamiltonian form of the equations of motion is
\begin{equation}
\begin{cases}
\dot{\bf q}_i=
m_i^{-1}{\bf p}_i,\cr
\dot{\bf p}_i=
\sum_{j=1, j\ne i}^n{m_im_j[{\bf q}_j-({\bf q}_i\cdot{\bf q}_j){\bf q}_i]\over
[1-({\bf q}_i\cdot{\bf q}_j)^2]^{3/2}}-m_i^{-1}
({\bf p}_i\cdot{\bf p}_i){\bf q}_i,\cr
{\bf q}_i\cdot{\bf q}_i=1, 
\ \  {\bf q}_i\cdot{\bf p}_i=0,
\ \ \kappa\ne 0,
\ \  i=1,\dots,n.\label{HamS}
\end{cases}
\end{equation}
Consequently the integral of energy has the form
\begin{equation}
\sum_{i=1}^nm_i^{-1}({\bf p}_i\cdot{\bf p}_i)-\sum_{i=1}^n\sum_{j=1,j\ne i}^n\frac{m_im_j~ \frac{{\bf q}_i\cdot{\bf q}_j}{\snorm}}{
\sqrt{1-\left(\frac{{\bf q}_i\cdot{\bf q}_j}{\snorm}\right)^2}}=2h,
\label{eneS}
\end{equation}
which, via ${\bf q}_i\cdot{\bf q}_i=1,\ i=1,\dots,n$, becomes
\begin{equation}
\sum_{i=1}^nm_i^{-1}({\bf p}_i\cdot{\bf p}_i)-\sum_{i=1}^n\sum_{j=1,j\ne i}^n
{m_im_j\q_i\cdot\q_j\over\sqrt{1-(\q_i\cdot\q_j)^2}}=2h,
\label{enerS}
\end{equation}
and the integrals of the angular momentum take the form
\begin{equation}
\sum_{i=1}^n{\bf q}_i\times{\bf p}_i={\bf c}.
\end{equation}
Notice that sometimes we can use the simpler form \eqref{enerS} of the energy integral, but whenever we need to exploit the homogeneity of the potential or have to differentiate it, we must use the more complicated form \eqref{eneS}.


\subsection{Equations of motion in ${\bf H}^2$}

In this case, the force function (\ref{pothom}) takes the form
\begin{equation}
U_{-1}({\bf q})=-{1\over 2}\sum_{i=1}^n\sum_{j=1,j\ne i}^n\frac{m_im_j~\frac{\q_i\boxdot\q_j}{\hnorm}}{\sqrt{\left(\frac{\q_i\boxdot\q_j}{\hnorm}   \right)^2-1}}
,\label{potH}
\end{equation}
so the equations of motion and their constraints become
\begin{multline}
m_i\ddot{\bf q}_i=\overline\nabla_{{\bf q}_i}U_{-1}({\bf q})+m_i(\dot{\bf q}_i\boxdot\dot{\bf q}_i){\bf q}_i,\ {\bf q}_i\boxdot{\bf q}_i=-1,\  {\bf q}_i\boxdot\dot{\bf q}_i=0,\\ 
i=1,\dots,n.\label{eqH}
\end{multline}

In terms of coordinates, the equations of motion and their constraints can be written as
\begin{equation}
\begin{cases}
m_i\ddot{x}_i=\ \ {\partial U_{-1}\over\partial x_i}+m_i(\dot{x}_i^2+\dot{y}_i^2-\dot{z}_i^2)x_i,\cr
m_i\ddot{y}_i=\ \  {\partial U_{-1}\over\partial y_i}+m_i(\dot{x}_i^2+\dot{y}_i^2-\dot{z}_i^2)y_i,\cr
m_i\ddot{z}_i=-{\partial U_{-1}\over\partial z_i}+m_i(\dot{x}_i^2+\dot{y}_i^2-\dot{z}_i^2)z_i,\cr
x_i^2+y_i^2-z_i^2=-1,\ \ x_i\dot{x}_i+y_i\dot{y}_i-z_i\dot{z}_i=0,\ \ i=1,\dots,n,
\label{eqcoordH}
\end{cases}
\end{equation}
and by computing the gradients they become
\begin{equation}
\begin{cases}
\ddot{x}_i=\sum_{j=1,j\ne i}^n{m_j{x_j+{x_ix_j+y_iy_j-z_iz_j\over -x_i^2-y_i^2+z_i^2}x_i
\over \sqrt{-x_i^2-y_i^2+z_i^2}\sqrt{-x_j^2-y_j^2+z_j^2}}\over
\bigg[\Big({x_ix_j+y_iy_j-z_iz_j\over\sqrt{-x_i^2-y_i^2+z_i^2}\sqrt{-x_j^2-y_j^2+z_j^2}}\Big)^2-1\bigg]^{3/2}}+(\dot{x}_i^2+\dot{y}_i^2-\dot{z}_i^2)x_i,\cr
\ddot{y}_i=\sum_{j=1,j\ne i}^n{m_j{y_j+{x_ix_j+y_iy_j-z_iz_j\over -x_i^2-y_i^2+z_i^2}y_i
\over \sqrt{-x_i^2-y_i^2+z_i^2}\sqrt{-x_j^2-y_j^2+z_j^2}}\over
\bigg[\Big({x_ix_j+y_iy_j-z_iz_j\over\sqrt{-x_i^2-y_i^2+z_i^2}\sqrt{-x_j^2-y_j^2+z_j^2}}\Big)^2-1\bigg]^{3/2}}+(\dot{x}_i^2+\dot{y}_i^2-\dot{z}_i^2)y_i,\cr
\ddot{z}_i=\sum_{j=1,j\ne i}^n{m_j{z_j+{x_ix_j+y_iy_j-z_iz_j\over -x_i^2-y_i^2+z_i^2}z_i
\over \sqrt{-x_i^2-y_i^2+z_i^2}\sqrt{-x_j^2-y_j^2+z_j^2}}\over
\bigg[\Big({x_ix_j+y_iy_j-z_iz_j\over\sqrt{-x_i^2-y_i^2+z_i^2}\sqrt{-x_j^2-y_j^2+z_j^2}}\Big)^2-1\bigg]^{3/2}}+(\dot{x}_i^2+\dot{y}_i^2-\dot{z}_i^2)z_i,\cr
x_i^2+y_i^2-z_i^2=-1,\ \ x_i\dot{x}_i+y_i\dot{y}_i-z_i\dot{z}_i=0,\ \ i=1,\dots,n.\label{coordHfull}
\end{cases}
\end{equation}
For the same reasons described in the previous subsection, we can use the constraints to write from now on the above system as
\begin{equation}
\begin{cases}
\ddot{x}_i=\sum_{j=1,j\ne i}^n{m_j[x_j+(x_ix_j+y_iy_j-z_iz_j)x_i]\over
[(x_ix_j+y_iy_j-z_iz_j)^2-1]^{3/2}}+(\dot{x}_i^2+\dot{y}_i^2-\dot{z}_i^2)x_i,\cr
\ddot{y}_i=\sum_{j=1,j\ne i}^n{m_j[y_j+(x_ix_j+y_iy_j-z_iz_j)y_i]\over
[(x_ix_j+y_iy_j-z_iz_j)^2-1]^{3/2}}+(\dot{x}_i^2+\dot{y}_i^2-\dot{z}_i^2)y_i,\cr
\ddot{z}_i=\sum_{j=1,j\ne i}^n{m_j[z_j+(x_ix_j+y_iy_j-z_iz_j)z_i]\over
[(x_ix_j+y_iy_j-z_iz_j)^2-1]^{3/2}}+(\dot{x}_i^2+\dot{y}_i^2-\dot{z}_i^2)z_i,\cr
x_i^2+y_i^2-z_i^2=-1,\ \ x_i\dot{x}_i+y_i\dot{y}_i-z_i\dot{z}_i=0,\ \ i=1,\dots,n.\label{coordH}
\end{cases}
\end{equation}
The Hamiltonian form of the equations of motion is
\begin{equation}
\begin{cases}
\dot{\bf q}_i=
m_i^{-1}{\bf p}_i,\cr
\dot{\bf p}_i=
\sum_{j=1, j\ne i}^n{m_im_j[{\bf q}_j+({\bf q}_i\boxdot{\bf q}_j){\bf q}_i]\over
[({\bf q}_i\boxdot{\bf q}_j)^2-1]^{3/2}}+m_i^{-1}
({\bf p}_i\boxdot{\bf p}_i){\bf q}_i,\cr
{\bf q}_i\boxdot{\bf q}_i=-1, 
\ \  {\bf q}_i\boxdot{\bf p}_i=0,
\ \ \kappa\ne 0,
\ \  i=1,\dots,n.\label{HamH}
\end{cases}
\end{equation}
Consequently the integral of energy takes the form
\begin{equation}
\sum_{i=1}^nm_i^{-1}({\bf p}_i\boxdot{\bf p}_i)+\sum_{i=1}^n\sum_{j=1,j\ne i}^n\frac{m_im_j~\frac{\q_i\boxdot\q_j}{\hnorm}}{\sqrt{\left(\frac{\q_i\boxdot\q_j}{\hnorm}   \right)^2-1}}=2h,
\label{eneH}
\end{equation}
which, via ${\bf q}_i\boxdot{\bf q}_i=-1, \ i=1,\dots,n$, becomes
\begin{equation}
\sum_{i=1}^nm_i^{-1}({\bf p}_i\boxdot{\bf p}_i)+\sum_{i=1}^n\sum_{j=1,j\ne i}^n
{m_im_j\q_i\boxdot\q_j\over\sqrt{(\q_i\boxdot\q_j)^2-1}}=2h,
\label{enerH}
\end{equation}
and the integrals of the angular momentum can be written as
\begin{equation}
\sum_{i=1}^n{\bf q}_i\boxtimes{\bf p}_i={\bf c}.
\end{equation}
Notice that sometimes we can use the simpler form \eqref{enerH} of the energy integral, but whenever we need to exploit the homogeneity of the potential or have to differentiate it, we must use the more complicated form \eqref{eneH}.


\subsection{Equations of motion in ${\bf S}^{\mu}$ and ${\bf H}^{\mu}$}

The formalism we adopted in this paper allows a straightforward generalization 
of the $n$-body problem to ${\bf S}^{\mu}$ and ${\bf H}^{\mu}$ for any integer 
$\mu\ge 1$. The equations of motion in $\mu$-dimensional spaces of constant
curvature have the form (\ref{eqmotion}) for vectors ${\bf q}_i$ and ${\bf q}_j$ 
of ${\bf R}^{\mu+1}$ constrained to the corresponding manifold. 
It is then easy to see from any coordinate-form
of the system that ${\bf S}^\nu$ and ${\bf H}^\nu$ are invariant sets for the
equations of motion in ${\bf S}^{\mu}$ and ${\bf H}^{\mu}$, respectively, for any
integer $\nu<\mu$.

Indeed, this is the case, say, for equations \eqref{coordS}, if we take $x_i(0)=0, \dot{x}_i(0)=0, \ i=1,\dots,n$. Then the equations of $\ddot{x}_i$ are identically 
satisfied, and the motion takes place on the circle $y^2+z^2=1$. The generalization 
of this idea from one component to any number $\nu$ of components in a
$(\mu+1)$-dimensional space, with $\nu<\mu$, is straightforward. Therefore 
the study of the $n$-body problem on surfaces of constant curvature is fully justified.

The only aspect of this generalization that is not obvious from our formalism 
is how to extend the cross product to higher dimensions. But this extension 
can be done as in general relativity with the help of the exterior product. However, 
we will not get into higher dimensions in this paper. Our further goal is to study the
2-dimensional case.


\section{Singularities}

Singularities have always been a rich source of research in the
theory of differential equations. The $n$-body problem we derived
in the previous section seems to make no exception from this rule. 
In what follows, we will point out the singularities that 
occur in our problem and prove some results related to them.
The most surprising seems to be the existence of a class of 
solutions with some hybrid singularities, which are both collisional and
non-collisional.


\subsection{Singularities of the equations}

The equations of motion (\ref{Ham}) have restrictions. First, the variables are 
constrained to a surface of constant curvature, i.e.\ $({\bf q},{\bf p})\in {\bf T}^*({\bf M}_\kappa^2)^n$, where ${\bf M}^2_\kappa$ is the surface of curvature $\kappa\ne 0$ (in particular, ${\bf M}^2_1={\bf S}^2$ and ${\bf M}^2_{-1}={\bf H}^2$), 
${\bf T}^*({\bf M}_\kappa^2)^n$ is the cotangent bundle of ${\bf M}^2_\kappa$. Second, system (\ref{Ham}),
which contains the gradient \eqref{gra}, is
undefined in the set ${\bf \Delta}:=\cup_{1\le i<j\le n}{\bf \Delta}_{ij}$, with
$${\bf \Delta}_{ij}:=\{{\bf q}\in({\bf M}^2_\kappa)^n\ |\ (\kappa{\bf q}_i\odot{\bf q}_j)^2=1\},$$ 
where both the force function \eqref{pothom} and its gradient \eqref{gra} become infinite.
Thus the set $\bf\Delta$ contains the singularities of the equations of motion.

The singularity condition, $(\kappa{\bf q}_i\odot{\bf q}_j)^2=1$, suggests that we
consider two cases, and thus write ${\bf \Delta}_{ij}={\bf \Delta}_{ij}^+\cup{\bf \Delta}_{ij}^-$, where
$$
{\bf \Delta}_{ij}^+:=\{{\bf q}\in({\bf M}^2_\kappa)^n\ |\ \kappa{\bf q}_i\odot{\bf q}_j=1\}\ \
{\rm and}\ \ 
{\bf \Delta}_{ij}^-:=\{{\bf q}\in({\bf M}^2_\kappa)^n\ |\ \kappa{\bf q}_i\odot{\bf q}_j=-1\}.
$$
Accordingly, we define 
$$
{\bf \Delta}^+:=\cup_{1\le i<j\le n}{\bf \Delta}_{ij}^+\ \ {\rm and}\ \
{\bf \Delta}^-:=\cup_{1\le i<j\le n}{\bf \Delta}_{ij}^-.
$$
Then ${\bf \Delta}={\bf \Delta}^+\cup{\bf \Delta}^-$. The elements of ${\bf \Delta}^+$ correspond to collisions for any $\kappa\ne 0$,
whereas the elements of ${\bf \Delta}^-$ correspond to what we will call antipodal singularities when $\kappa>0$. The latter occur when two bodies are at the opposite 
ends of the same diameter of a sphere. For $\kappa<0$, such singularities do not 
exist because $\kappa{\bf q}_i\odot{\bf q}_j\ge 1$. 

In conclusion, the equations of motion are undefined for configurations that involve
collisions on spheres or hyperboloids, as well as for configurations with antipodal
bodies on spheres of any curvature $\kappa>0$. In both cases, the gravitational forces become infinite. 

In the 2-body problem, ${\bf \Delta}^+$ and ${\bf \Delta}^-$ are disjoint sets. Indeed,
since there are only two bodies, $\kappa{\bf q}_1\cdot{\bf q}_2$ is either $+1$ or
$-1$, but cannot be both. The set ${\bf \Delta}^+\cap{\bf \Delta}^-$, however, is not empty for $n\ge 3$. In the 3-body problem, for instance, the configuration in which 
two bodies are at collision and the third lies at the opposite end of the corresponding diameter is, what we will call from now on, a collision-antipodal singularity.

The theory of differential equations merely regards singularities as points for which 
the equations break down, and must therefore be avoided. But singularities
exhibit sometimes a dynamical structure. In the $3$-body problem in $\bf R$, for instance, the set of binary collisions is attractive in the sense that for any given initial velocities, there are initial positions such that if two bodies come close
enough to each other but far enough from other collisions, then the collision
will take place. (Things are more complicated with triple collisions. Two of the
bodies coming close to triple collisions may form a binary while the third
gets expelled with high velocity away from the other two, \cite{McGe}.) 

Something similar happens for binary collisions in the 3-body 
problem on a geodesic of ${\bf S}^2$. Given some initial velocities, one 
can choose initial positions that put $m_1$ and $m_2$ close enough to a 
binary collision, and $m_3$ far enough from an antipodal singularity with 
either $m_1$ or $m_2$, such that the binary collision takes place. This 
is indeed the case because the attraction between $m_1$ and $m_2$ can 
be made as large as desired by placing the bodies close enough to each 
other. Since $m_3$ is far enough from an antipodal position, and no
comparable force can oppose the attraction between $m_1$ and $m_2$,
these bodies will collide.

Antipodal singularities lead to a new phenomenon on geodesics of ${\bf S}^2$. Given initial velocities, no matter how close one chooses initial positions near an antipodal singularity, the corresponding solution is repelled in future time from this singularity as long as no collision force compensates for this force. So while binary collisions can be regarded as attractive if far away from binary antipodal singularities, binary antipodal singularities can be seen as repulsive if far away from collisions. But what happens when collision and antipodal singularities are close to each other? As we will see in the next subsection, the behavior of solutions in that region is sensitive to the choice of masses and initial conditions. In particular, we will prove the existence of some hybrid singular solutions in the 3-body problem, namely those that end in finite time in a collision-antipodal singularity, as well as of solutions
that reach a collision-antipodal configuration but remain analytic at this point.


\subsection{Solution singularities}

The set $\bf\Delta$ is related to singularities which arise from the question 
of existence and uniqueness of initial value problems. For initial conditions 
$({\bf q},{\bf p})(0)\in{\bf T}^*({\bf M}_\kappa^2)^n$ with ${\bf q}(0)\notin\bf\Delta$, standard results of the theory of differential equations ensure 
local existence and uniqueness of an analytic solution $({\bf q},{\bf p})$ defined 
on some interval $[0,t^+)$. Since the surfaces ${\bf M}^2_\kappa$ are connected, 
this solution can be analytically extended to an interval $[0,t^*)$, with $0<t^+\le t^*\le\infty$. If $t^*=\infty$, the solution is globally defined. But if $t^*<\infty$, the solution is called singular, and we say that it has a singularity at time $t^*$. 

There is a close connection between singular solutions and singularities of the
equations of motion. In the classical case ($\kappa=0$), this connection was pointed
out by Paul Painlev\'e towards the end of the 19th century. In his famous lectures
given in Stockholm, \cite{Pai}, he showed that every singular solution $({\bf q},{\bf p})$
is such that ${\bf q}(t)\to{\bf\Delta}$ when $t\to t^*$, for otherwise the solution would
be globally defined. In the Euclidean case, $\kappa=0$, the set $\bf\Delta$ is
formed by all configurations with collisions, so when ${\bf q}(t)$ tends to an element
of $\bf\Delta$, the solution ends in a collision singularity. But it is also possible that 
${\bf q}(t)$ tends to $\bf\Delta$ without asymptotic phase, i.e.~by oscillating among
various elements without ever reaching a definite position. Painlev\'e conjectured
that such noncollision singularities, which he called pseudocollisions, exist. 
In 1908, Hugo von Zeipel showed that a necessary condition for a solution to
experience a pseudocollision is that the motion becomes unbounded in finite
time, \cite{Zei}, \cite{McG}. Zhihong (Jeff) Xia produced the first example of this 
kind in 1992, \cite{Xia}. Historical accounts of this development appear in
 \cite{Diac} and \cite{Dia0}.

The results of Painlev\'e don't remain intact in our problem, \cite{Dia1}, 
\cite{Dia2008}, so whether pseudocollisions exist for $\kappa\ne 0$ is not clear. Nevertheless, we will now show that there are solutions ending in collision-antipodal singularities of the equations of motion, solutions these singularities repel, as well
as solutions that are not singular at such configurations. To prove these facts, we need the result stated below, which provides a criterion for determining the direction of motion along a great circle in the framework of an isosceles problem defined in
an invariant set ${\bf S}^1$.


\begin{lemma}
Consider the $n$-body problem in ${\bf S}^2$, and assume that a body of
mass $m$ is at rest at time $t_0$ on the geodesic $z=0$ within its first quadrant, $x,y>0$. Then, if

(a) $\ddot{x}(t_0)> 0$ and $\ddot{y}(t_0)< 0$, the force pulls
the body along the circle toward the point $(x,y)=(1,0)$.

(b) $\ddot{x}(t_0)< 0$ and $\ddot{y}(t_0)> 0$, the force pulls
the body along the circle toward the point $(x,y)=(0,1)$.

(c) $\ddot{x}(t_0)\le 0$ and $\ddot{y}(t_0)\le 0$, the force pulls the body toward the point $(1,0)$ if $\ddot{y}(t_0)/\ddot{x}(t_0)>y(t_0)/x(t_0)$, toward $(0,1)$ if $\ddot{y}(t_0)/\ddot{x}(t_0)<y(t_0)/x(t_0)$, but no force acts on the body if neither of the previous inequalities
holds. 

(d) $\ddot{x}(t_0)> 0$ and $\ddot{y}(t_0)> 0$, the motion is impossible.\label{singlemma}
\end{lemma}
\def\JPicScale{0.5}
\ifx\JPicScale\undefined\def\JPicScale{1}\fi
\unitlength \JPicScale mm
\begin{figure}
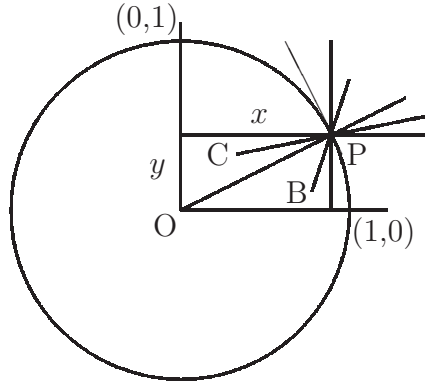


\caption{The relative positions of the force acting on $m$, while the body
is on the geodesic $z=0$.}\label{circ}
\end{figure}
\begin{proof}
By equation \eqref{h1}, $x\ddot x+y\ddot y=-(\dot x^2+\dot y^2) \le 0$, which means that the force acting on $m$ is always directed along the tangent at $m$ to the geodesic circle $z=0$ or inside the half-plane containing this circle. Assuming that an $xy$-coordinate system is fixed at the origin of the acceleration vector (point P in Figure \ref{circ}), this vector always lies in the half-plane below the line of slope $-x(t_0)/y(t_0)$ (i.e.~the tangent to the circle at the point P in Figure \ref{circ}). We 
further prove each case separately.

(a) If $\ddot{x}(t_0)> 0$ and $\ddot{y}(t_0)<0$, the force acting on $m$ is 
represented by a vector that lies in the region given by the
intersection of the fourth quadrant (counted counterclockwise) and the half 
plane below the line of slope $-x(t_0)/y(t_0)$. Then, obviously, the force 
pulls the body along the circle in the direction of the point $(1,0)$.

(b) If $\ddot{x}(t_0)< 0$ and $\ddot{y}(t_0)> 0$, the force
acting on $m$ is represented by a vector that lies in the region given by the
intersection of the second quadrant and the half plane lying below the line of 
slope $-x(t_0)/y(t_0)$. Then, obviously, the force pulls the body along 
the circle in the direction of the point $(0,1)$.

(c) If $\ddot{x}(t_0)\le 0$ and $\ddot{y}(t_0)\le 0$, the force acting on $m$ is represented by a vector lying in the third quadrant. Then the direction in which
this force acts depends on whether the acceleration vector lies: (i) below the line of slope $y(t_0)/x(t_0)$ (PB is below OP in Figure \ref{circ}); (ii) above it (PC is above OP); or (iii) on it (i.e.~on the line OP). Case (iii) includes the case when the 
acceleration is zero.

In case (i), the acceleration vector lies on a line whose slope is larger
than $y(t_0)/x(t_0)$, i.e. $\ddot{y}(t_0)/\ddot{x}(t_0)>y(t_0)/x(t_0)$, so
the force pulls $m$ toward $(1,0)$. In case (ii), the acceleration vector 
lies on a line of slope that is smaller than $y(t_0)/x(t_0)$, i.e.~$\ddot{y}(t_0)/\ddot{x}(t_0)<y(t_0)/x(t_0)$, so the force pulls $m$ toward $(0,1)$. In case (iii), 
the acceleration vector is either zero or lies on the line of slope 
$y(t_0)/x(t_0)$, i.e.~$\ddot{y}(t_0)/\ddot{x}(t_0)=y(t_0)/x(t_0)$. But
the latter alternative never happens. This fact follows from the
equations of motion \eqref{eqmotion}, which show that the acceleration
is the difference between the gradient of the force function and a multiple
of the position vector. But according to Euler's formula for homogeneous 
functions, \eqref{eul}, and the fact that the velocities are zero, these vectors 
are orthogonal, so their difference can have the same direction as one of them only
if it is zero. This vectorial argument agrees with the kinematic facts, 
which show that if $\dot{x}(t_0)=\dot{y}(t_0)=0$ and the acceleration 
has the same direction as the position vector, then $m$ doesn't move, 
so $\dot{x}(t)=\dot{y}(t)=0$, and therefore $\ddot{x}(t)=\ddot{y}(t)=0$ for 
all $t$. In particular, this means that when $\ddot{y}(t_0)=\ddot{x}(t_0)=0$, 
no force acts on $m$, so the body remains fixed.

(d) If $\ddot{x}(t_0)> 0$ and $\ddot{y}(t_0)> 0$, the force acting on $m$ is 
represented by a vector that lies in the region given by the intersection 
between the first quadrant and the half-plane lying below the line of slope 
$-x(t_0)/y(t_0)$. But this region is empty, so the motion doesn't take place.
\end{proof}

We will further prove the existence of solutions with collision-antipodal
singularities, solutions repelled from collision-antipodal singularities in
positive time, as well as of solutions that remain analytic at a collision-antipodal
configuration. They show that the dynamics of $\bf\Delta^+\cap\Delta^-$ is 
more complicated than the dynamics of $\bf\Delta^+$ and $\bf\Delta^-$ away
from the intersection, since solutions can go both towards and away from 
this set for $t>0$, and can even avoid singularities. This result represents a first example of a non-collision singularity reached by only three bodies as well as a first
example of a non-singularity collision.


\begin{theorem}
Consider the 3-body problem in ${\bf S}^2$ with the bodies $m_1$ and $m_2$ 
having mass $M>0$ and the body $m_3$ having mass $m>0$. Then 

(i) there are values of\ \ $m$ and $M$, as well as initial conditions, for which the solutions end in finite time in a collision-antipodal singularity;

(ii) other choices of masses and initial conditions lead to solutions that are repelled from a collision-antipodal singularity;

(iii) and yet other choices of masses and initial data correspond to solutions
that reach a collision-antipodal configuration but remain analytic at this point.

\label{singularity}
\end{theorem}

\def\JPicScale{0.5}
\ifx\JPicScale\undefined\def\JPicScale{1}\fi
\unitlength \JPicScale mm
\begin{figure}
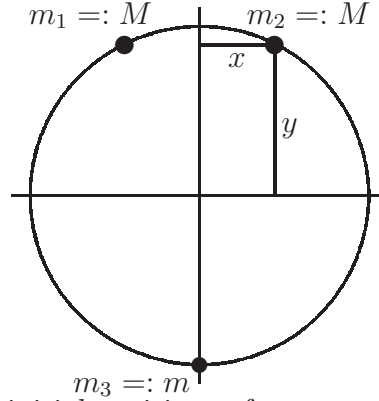


\caption{The initial positions of $m_1, m_2$, and $m_3$ 
on the geodesic $z=0$.}\label{cir}
\end{figure}

\begin{proof}
Let us start with some initial conditions we will refine on the way. During the 
refinement process, we will also choose suitable masses. Consider
\begin{align*}
x_1(0)&=-x(0),& y_1(0)&=y(0),&  z_1(0)&=0,\\
x_2(0)&=x(0),&  y_2(0)&=y(0),&  z_2(0)&=0,\\
x_3(0)&=0,&  y_3(0)&=-1,&   z_3(0)&=0,
\end{align*}
as well as zero initial velocities, where $0<x(t),y(t)<1$ are functions with $x(t)^2+y(t)^2=1$. Since all $z$ coordinates are zero, only the equations of 
coordinates $x$ and $y$ play a role in the motion. The symmetry of these initial 
conditions implies that $m_3$ remains fixed for all time (in fact the equations corresponding to $\ddot{x}_3$ and $\ddot{y}_3$ reduce to identities), that the
angular momentum is zero, and that it is enough to see what happens for $m_2$, because $m_1$ behaves symmetrically with respect to the $y$ axis. Thus, 
substituting the above initial conditions into the equations of motion, we obtain
\begin{equation}
\ddot{x}(0)=-{y(0)\over x^2(0)}\bigg({M\over 4y^2(0)}-m\bigg)\ \ \ {\rm and}\ \ \
\ddot{y}(0)={1\over x(0)}\bigg({M\over 4y^2(0)}-m\bigg).\label{incond}
\end{equation}
These equations show that several situations occur, depending on the
choice of masses and initial positions. Here are two significant possibilities.

1. For $M\ge4m$, it follows that $\ddot x(0)<0$ and $\ddot y(0)>0$ for any 
choices of initial positions with $0<x(0),y(0)<1$.

2. For $M<4m$, there are initial positions for which: 

\hskip0.6cm (a) $\ddot x(0)<0$ and $\ddot y(0)>0$,

\hskip0.6cm (b) $\ddot x(0)>0$ and $\ddot y(0)<0$,

\hskip0.6cm (c) $\ddot x(0)=\ddot y(0)=0$.

In case 2(c), the solutions are fixed points of the equations of motion, a situation
achieved, for instance, when $M=2m$ and $x(0)=y(0)=\sqrt{2}/2$. The cases of interest
for us, however, are 1 and 2(b). In the former, $m_2$ begins to move from rest
towards a collision with $m_1$ at $(0,1)$, but whether this collision takes place
also depends on velocities, which affect the equations of motion. In the latter case,
$m_2$ moves away from the same collision, and we need to see again how the
velocities alter this initial tendency. So let us write now the equations of motion for $m_2$ starting from arbitrary masses $M$ and $m$. The computations lead
us to the system
\begin{equation}
\begin{cases}
\ddot x=-{M\over 4x^2y}+{my\over x^2}-(\dot{x}^2+\dot{y}^2)x\cr
\ddot y={M\over 4xy^2}-{m\over x}-(\dot{x}^2+\dot{y}^2)y\cr
\end{cases}\label{initcond}
\end{equation}
and the energy integral 
$$
\dot{x}^2+\dot{y}^2={h\over M}-{2my\over x}+{M(2y^2-1)\over 2xy}.
$$
Substituting this expression of $\dot{x}^2+\dot{y}^2$ into equations (\ref{initcond}),
we obtain 
\begin{equation}
\begin{cases}
\ddot x={4(M-2m)x^4-2(M-2m)x^2-M+4m\over 4x^2y}-{h\over M}x\cr
\ddot y={M+2(M-2m)y^2-4(M-2m)y^4\over 4xy^2}-{h\over M}y.\cr
\end{cases}\label{twoeq}
\end{equation}
We will further focus on the first class of orbits announced in this theorem.

(i) To prove the existence of solutions with collision-antipodal singularities, let 
us further examine the case $M=8m$, which brings system \eqref{twoeq} to the form
\begin{equation}
\begin{cases}
\ddot x={6mx^2\over y}-{3m\over y}-{m\over x^2y}-{h\over 8m}x\cr
\ddot y={2m\over xy^2}+{3m\over x}-{6my^2\over x}-{h\over 8m}y,\cr
\end{cases}\label{coll}
\end{equation}
with the energy integral
\begin{equation}
\dot{x}^2+\dot{y}^2+{4mx\over y}-{2my\over x}={h\over 8m}.\label{enecoll}
\end{equation}
Then, as $x\to 0$ and $y\to 1$, both $\ddot x$ and $\ddot y$ tend to $-\infty$,
so they are ultimately negative. This fact corresponds to case (c) of
Lemma \ref{singlemma}. But a simple computation shows that $\ddot y/\ddot x$
tends to zero as $x\to 0$ and $y\to 1$. Since $y/x>0$, it follows that if 
$(x(0),y(0))$ is chosen close enough to $(0,1)$, then $\ddot y(0)/\ddot x(0)<
y(0)/x(0)$, so according to the conclusion of Lemma \ref{singlemma}(c) the collision-antipodal configuration is reached. As the forces and the potential are infinite at this point, using the energy relation \eqref{enecoll}  it follows that the velocities are also infinite. Consequently the motion cannot be analytically extended beyond the
collision-antipodal configuration, which thus proves to be a singularity.

(ii) To show the existence of solutions repelled from a collision-antipodal 
singularity of the equations of motion in positive time, let us take $M=2m$. Then equations
\eqref{twoeq} have the form
\begin{equation}
\begin{cases}
\ddot x={m\over 2x^2y}-{h\over 2m}x\cr
\ddot y={m\over 2xy^2}-{h\over 2m}y,\cr
\end{cases}\label{repel}
\end{equation}
with the integral of energy 
\begin{equation}
\dot x^2+\dot y^2+{m\over xy}={h\over 2m},\label{en}
\end{equation}
which implies that $h>0$. Obviously, as $x\to 0$ and $y\to 1$, the forces and and
the kinetic energy become infinite, so the collision-antipodal configuration is a
singularity if it were to be reached. But as we will further see, this cannot happen
for this choice of masses. Indeed, as we saw in case 2(c) above, the initial position $x(0)=$ $y(0)=\sqrt{2}/2$ corresponds to a fixed point of the equations of motion for zero initial velocities. Therefore we must seek the desired solution for initial conditions with $0<x(0)<\sqrt{2}/2$ and the corresponding choice of $y(0)>0$. Let us pick any such initial positions, as close to the collision-antipodal singularity as we want, and zero
initial velocities. For $x\to 0$, however, equations \eqref{repel} show that both $\ddot x$ and $\ddot y$ grow positive. But according to case (d) of Lemma \ref{singlemma}, such an outcome is impossible, so the motion cannot come infinitesimally close to the corresponding collision-antipodal singularity, which 
repels any solution with $M=2m$ and initial conditions chosen as we previously described.

(iii) To prove the existence of solutions that have no singularity at a collision-antipodal configuration, let us further examine the case $M=4m$, which brings system \eqref{twoeq} to the form 
\begin{equation}
\begin{cases}
\ddot{x}={m(2x^2-1)\over y}-{h\over 4m}x\cr
\ddot{y}={mx(2y^2+1)\over y^2}-{h\over 4m}y.\cr
\end{cases}
\label{yeq}
\end{equation}
For this choice of masses, the energy integral becomes
\begin{equation}
\dot{x}^2+\dot{y}^2+{2mx\over y}={h\over 4m}.\label{ene}
\end{equation}
We can compute the value of $h$ from the initial conditions. Thus, for initial 
positions $x(0), y(0)$ and initial velocities $\dot x(0)=\dot y(0)=0$, the energy 
constant is $h=8m^2x(0)/y(0)>0$. 

Assuming that $x\to 0$ and $y\to 1$, equations (\ref{yeq}) imply that $\ddot x(t)\to -m<0$ and $\ddot y(t)\to -h/4m<0$, which means that the forces are finite at the collision-antipodal configuration. We are thus in the case (c) of
Lemma \ref{singlemma}, so to determine the direction of motion for $m_2$ when
it comes close to $(0,1)$, we need to take into account the ratio ${\ddot y/ \ddot x}$, which tends to $h/4m^2$ as $x\to 0$. Since $h=8m^2x(0)/y(0)$, $\lim_{x\to 0}({\ddot y/ \ddot x})=2x(0)/y(0)$. Then $2x(0)/y(0)<y(0)/x(0)$ for any $x(0)$ and $y(0)$ with $0<x(0)<1/\sqrt{3}$ and the corresponding choice of $y(0)>0$ given by the constraint $x^2(0)+y^2(0)=1$. But the inequality $2x(0)/y(0)<y(0)/x(0)$ is equivalent to the condition $\ddot{y}(t_0)/\ddot{x}(t_0)<y(t_0)/x(t_0)$ in Lemma \ref{singlemma}(c), according to which the force pulls $m_2$ toward $(0,1)$. Therefore the velocity and the force acting on $m_2$ keep this body on the same path until the collision-antipodal configuration occurs. 

It is also clear from equation \eqref{ene} that the velocity is positive and finite at collision. Since the distance between the initial position and $(0,1)$ is bounded, $m_2$ collides with $m_1$ in finite time. Therefore the choice of masses with $M=4m$, initial positions $x(0),y(0)$ with $0<x(0)<1/\sqrt{3}$ and the corresponding value of $y(0)$, and initial velocities $\dot{x}(0)=\dot{y}(0)=0$, leads to a solution that remains analytic at the collision-antipodal configuration, so the motion naturally
extends beyond this point.
\end{proof}


\section{ Relative equilibria in ${\bf S}^2$}

In this section we will prove a few results related to fixed points and elliptic relative equilibria in ${\bf S}^2$. Since, by Euler's theorem (see Appendix), every element of the group $SO(3)$ can be written, in an orthonormal basis, as a rotation about the $z$ axis, we can define elliptic relative equilibria as follows.
\begin{definition} 
An elliptic relative equilibrium in ${\bf S}^2$ is a solution of the form ${\bf q}_i=(x_i,y_i,z_i)$, $i=1,\dots, n$, of equations (\ref{coordS}) with $x_i=r_i\cos(\omega t+\alpha_i), y_i=r_i\sin(\omega t+\alpha_i), z_i={\rm constant},$ where $\omega, \alpha_i,$ and $r_i$, with $0\le r_i=(1-z_i^2)^{1/2}\le 1,\ i=1,\dots,n,$ are constants.\label{reS}
\end{definition}

Notice that although the equations of motion don't have an integral of the
center of mass, a ``weak'' property of this kind occurs for elliptic relative
equilibria. Indeed, it is easy to see that if all the bodies are at all times 
on one side of a plane containing the rotation axis, then the integrals of the angular momentum are violated. This happens because under such circumstances the vector representing the total angular momentum cannot be zero or parallel to the
$z$ axis.


\subsection{Fixed points}
The simplest solutions of the equations of motion are fixed points. They
can be seen as trivial relative equilibria that correspond to $\omega=0$. In
terms of the equations of motion, we can define them as follows.
\begin{definition}
A solution of system (\ref{HamS}) is called a fixed point if 
$$\nabla_{{\bf q}_i}U_1({\bf q})(t)={\bf p}_i(t)={\bf 0}\ \ {\rm for} \ {\rm all}\ \
t\in{\bf R}\ \ {\rm and}\ \ i=1,\dots,n.$$
\end{definition}
Let us start with finding the simplest fixed points, namely those that occur
when all the masses are equal.


\begin{theorem}
Consider the $n$-body problem in ${\bf S}^2$ with $n$ odd. If the masses are all equal, the regular $n$-gon lying on any geodesic is a fixed point of the equations of motion. For $n=4$, the regular tetrahedron is a fixed point too. \label{fix}
\end{theorem}
\begin{proof}
Assume that $m_1=m_2=\dots =m_n$, and consider an $n$-gon with an odd number
of sides inscribed in a geodesic of ${\bf S}^2$ with a body, initially at rest, at each 
vertex. In general, two forces act on the body of mass $m_i$: the force $\nabla_{{\bf q}_i}U_1({\bf q})$, which is due to the interaction with the other bodies, and the force $-m_i(\dot{\bf q}_i\cdot\dot{\bf q}_i){\bf q}_i$, which is due to the constraints. The latter force is zero at $t=0$ because the bodies are initially at rest. Since ${\bf q}_i\cdot\nabla_{{\bf q}_i}U_1({\bf q})=0$,
it follows that $\nabla_{{\bf q}_i}U_1({\bf q})$ is orthogonal to ${\bf q}_i$, and thus tangent to ${\bf S}^2$. Then the
symmetry of the $n$-gon implies that, at the initial moment $t=0$, $\nabla_{{\bf q}_i}U_1({\bf q})$ is the sum of pairs of forces, each pair consisting of opposite
forces that cancel each other. This means that $\nabla_{{\bf q}_i}U_1({\bf q})={\bf 0}$. Therefore, from the equations of motion and the fact that
the bodies are initially at rest, it follows that 
$$
\ddot{\bf q}_i(0)=-(\dot{\bf q}_i(0)\cdot\dot{\bf q}_i(0)){\bf q}_i(0)={\bf 0}, \ \ i=1,\dots,n.
$$
But then no force acts on the body of mass $m_i$ at time $t=0$, consequently the body doesn't move. So $\dot{\bf q}_i(t)={\bf 0}$ for all $t\in{\bf R}$. Then $\ddot{\bf q}_i(t)={\bf 0}$
for all $t\in{\bf R}$, therefore $\nabla_{{\bf q}_i}U_1({\bf q})(t)={\bf 0}$ for all $t\in{\bf R}$,
so the $n$-gon is a fixed point of equations (\ref{coordS}).

Notice that if $n$ is even, the $n$-gon has $n/2$ pairs of antipodal vertices. Since
antipodal bodies introduce singularities into the equations of motion, only the $n$-gons with an odd number of vertices are fixed points of equations (\ref{coordS}).

The proof that the regular tetrahedron is a fixed point can be merely done by
computing that 4 bodies of equal masses with initial coordinates given by
${\bf q}_1=(0,0,1), {\bf q}_2=(0,2\sqrt{2}/3, -1/3), 
{\bf q}_3=(-2/\sqrt{6},-\sqrt{2}/3,-1/3), {\bf q}_4=
(2/\sqrt{6},-\sqrt{2}/3,-1/3)$,
satisfy system (\ref{coordS}), or by noticing that the forces acting on each 
body cancel each other because of the involved symmetry. 
\end{proof}

\begin{remark}
If equal masses are placed at the vertices of the other four regular polyhedra: octahedron (6 bodies), cube (8 bodies), dodecahedron (12 bodies), and icosahedron (20 bodies), they do not form fixed points because antipodal singularities occur in 
each case.
\end{remark}

\begin{remark}
In the proof of Theorem \ref{singularity}, we discovered that if one body
has mass $m$ and the other two mass $M=2m$, then the isosceles triangle 
with the vertices at $(0,-1,0)$, $(-\sqrt{2}/2,\sqrt{2}/2,0)$, and $(\sqrt{2}/2,\sqrt{2}/2,0)$ is a fixed point. Therefore one might expect that fixed points can be found for any given masses. But, as formula \eqref{incond} shows, this is not the case. Indeed, if one body has mass $m$ and the other two have masses $M\ge4m$, there is no configuration (which must be isosceles due to symmetry) that corresponds to a fixed point since $\ddot x$ and $\ddot y$ are never zero. This observation proves that in the 3-body problem, there are choices of masses for which the equations of motion lack fixed points.\label{rem}
\end{remark}

The following statement is an obvious consequence of the proof given for
Theorem \ref{fix}.


\begin{cor}
Consider an odd number of equal bodies, initially at the vertices of a regular $n$-gon
inscribed in a great circle of ${\bf S}^2$, and assume that the solution generated from
this initial position maintains the same relative configuration for all times. Then,  for
all $t\in{\bf R}$, this solution satisfies the conditions $\nabla_{{\bf q}_i}U_1({\bf q}(t))={\bf 0},\ i=1,\dots,n$.\label{cor1}
\end{cor}

It is interesting to see that if the bodies are within a hemisphere (meaning half
a sphere and its geodesic boundary), fixed points do not occur if at least one body is
not on the boundary. Let us formally state and prove this result.


\begin{theorem}
Consider an initial nonsingular configuration of the $n$-body problem in ${\bf S}^2$ 
for which all bodies lie within a hemisphere, meant to include its geodesic
boundary, with at least one body not on this geodesic. Then this configuration
is not a fixed point.\label{nofixS}
\end{theorem}
\begin{proof}
Without loss of generality we can consider the initial configuration of the
bodies $m_1,\dots, m_n$ in the hemisphere $z\ge 0$, whose boundary 
is the geodesic $z=0$. Then at least one body has the smallest $z$
coordinate, and let $m_1$ be one of these bodies. Also, at least one
body has its $z$ coordinate positive, and let $m_2$ be one of them. 
Since all initial velocities are zero, only the mutual forces between
bodies act on $m_1$. Then, according to the equations of motion (\ref{eqcoordS}),
$m_1\ddot{z}_1(0)={\partial\over\partial z_1}U_1({\bf q}(0))$.
But as no body has its $z$ coordinate smaller than $z_1$, 
the terms contained in the expression of ${\partial\over\partial z_1}U_1({\bf q}(0))$ that involve interactions
between $m_1$ and $m_i$ are all larger than or equal to zero for
$i=3,4,\dots,n$, while the term involving $m_2$ is strictly positive. 
Therefore ${\partial\over\partial z_1}U_1({\bf q}(0))>0$, so $m_1$ 
moves upward the hemisphere. Consequently the initial configuration 
is not a fixed point.
\end{proof}

\subsection{Polygonal solutions} 
We will further show that fixed points lying on geodesics of spheres can generate
relative equilibria.


\begin{theorem}
Consider a fixed point given by the masses $m_1, m_2,\dots, m_n$ that
lie on a great circle of ${\bf S}^2$. Then for every nonzero angular velocity,
this configuration generates a relative equilibrium along the great circle.
\label{fixrel}
\end{theorem}
\begin{proof}
Without loss of generality, we assume that the great circle is the equator $z=0$
and that for some given masses $m_1, m_2,\dots, m_n$ there exist $\alpha_1,\alpha_2,\dots,\alpha_n$ such that the configuration ${\bf q}=({\bf q}_1,\dots, {\bf q}_n)$ given by ${\bf q}_i=(x_i,y_i, 0), i=1,\dots,n$, with
\begin{equation}
x_i=\cos(\omega t +\alpha_i), y_i=\sin(\omega t +\alpha_i), \ i=1,\dots, n,
\label{rot}
\end{equation}
is a fixed point for $\omega =0$. This configuration can also be 
interpreted as being ${\bf q}(0)$, i.e.~the solution ${\bf q}$ at $t=0$ for any
$\omega\ne 0$.
So we can conclude that $\nabla_{{\bf q}_i}U_1({\bf q}(0))={\bf 0},\ i=1,\dots,n$. 
But then, for $t=0$, the equations of motion \eqref{eqcoordS} reduce to
\begin{equation}
\begin{cases}
\ddot{x}_i=-(\dot{x}^2_i+\dot{y}^2_i)x_i\cr
\ddot{y}_i=-(\dot{x}^2_i+\dot{y}^2_i)y_i,
\label{tralala}
\end{cases}
\end{equation}
$i=1,\dots,n.$ 
Notice, however, that
$\dot{x}_i=-\omega\sin(\omega t+\alpha_i), \ddot{x}_i=-\omega^2\cos(\omega t+\alpha_i), \dot{y}_i=-\omega\cos(\omega t+\alpha_i)$, and
$\ddot{y}_i=-\omega^2\sin(\omega t+\alpha_i),$ therefore $\dot{x}_i^2+
\dot{y}_i^2=\omega^2$. Using these computations, it is easy to see that 
$\bf q$ given by \eqref{rot} is a solution of \eqref{tralala} for every $t$, 
so no forces due to the constraints act on the bodies, neither at $t=0$ nor later. Since $\nabla_{{\bf q}_i}U_1({\bf q}(0))={\bf 0},\ i=1,\dots,n$, it follows
that the gravitational forces are in equilibrium at the initial moment, so
no gravitational forces act on the bodies either. Consequently, the rotation
imposed by $\omega\ne 0$ makes the system move like a rigid body, so
the gravitational forces further remain in equilibrium, consequently
$\nabla_{{\bf q}_i}U_1({\bf q}(t))={\bf 0},\ i=1,\dots,n$, for all $t$.
Therefore $\bf q$ given by \eqref{rot} satisfies equations \eqref{eqcoordS}.
Then, by Definition \ref{reS}, $\bf q$ is an elliptic relative equilibrium.
\end{proof}

The following result shows that relative equilibria generated by fixed points
obtained from regular $n$-gons on a great circle of ${\bf S}^2$ can occur only 
when the bodies rotate along the great circle.


\begin{theorem}
Consider an odd number of equal bodies, initially at the vertices of a regular $n$-gon
inscribed in a great circle of ${\bf S}^2$. Then the only elliptic relative equilibria
that can be generated from this configuration are the ones that rotate in
the plane of the original great circle.\label{rengon}
\end{theorem}
\begin{proof} 
Without loss of generality, we can prove this result for the equator $z=0$. Consider
therefore an elliptic relative equilibrium solution of the form
\begin{equation}
x_i=r_i\cos(\omega t+\alpha_i), \  y_i=r_i\sin(\omega t+\alpha_i), \ z_i=\pm
(1-r_i^2)^{1/2},\label{check}
\end{equation}
$ i=1,\dots,n,$ with $+$ taken for $z_i>0$ and $-$ for $z_i<0$. The only condition we impose on this solution is that $r_i$ and $\alpha_i,\ i=1,\dots,n$, are chosen
such that the configuration is a regular $n$-gon inscribed in a moving great circle of 
${\bf S}^2$ at all times. 
Therefore the plane of the $n$-gon can have 
any angle with, say, the $z$-axis. This solution has the derivatives
$$\dot{x}_i=-r_i\omega\sin(\omega t+\alpha_i), \  \dot{y}_i=r_i\omega\cos(\omega t+\alpha_i), \ \dot{z}_i=0,\  i=1,\dots,n,$$
$$\ddot{x}_i=-r_i\omega^2\cos(\omega t+\alpha_i), \  \ddot{y}_i=-r_i\omega^2\sin(\omega t+\alpha_i), \  \ddot{z}_i=0,\ i=1,\dots,n.$$
Then 
$$\dot{x}_i^2+\dot{y}_i^2+\dot{z}_i^2=r_i^2\omega^2, \ i=1,\dots,n.$$
Since, by Corollary \ref{cor1}, any $n$-gon solution with $n$ odd satisfies the conditions 
$$\nabla_{{\bf q}_i}U_1({\bf q})={\bf 0},\ i=1,\dots,n,$$
system (\ref{coordS}) reduces to 
$$
\begin{cases}
\ddot{x}_i=-(\dot{x}_i^2+\dot{y}_i^2+\dot{z}_i^2)x_i,\cr
\ddot{y}_i=-(\dot{x}_i^2+\dot{y}_i^2+\dot{z}_i^2)y_i,\cr
\ddot{z}_i=-(\dot{x}_i^2+\dot{y}_i^2+\dot{z}_i^2)z_i,\ i=1,\dots,n.\cr
\end{cases}
$$
Then the substitution of (\ref{check}) into the above equations leads to:
$$
\begin{cases}
r_i(1-r_i^2)\omega^2\cos(\omega t+\alpha_i)=0,\cr
r_i(1-r_i^2)\omega^2\sin(\omega t+\alpha_i)=0,\ i=1,\dots,n.\cr
\end{cases}
$$
But assuming $\omega\ne 0$, this system is nontrivially satisfied if and only
if $r_i=1,$ conditions which are equivalent to $z_i=0,\ i=1,\dots,n.$
Therefore the bodies must rotate along the equator $z=0$.
\end{proof}

Theorem \ref{rengon} raises the question whether elliptic relative equilibria given
by regular polygons can rotate on other curves than geodesics. The answer
is given by the following result.


\begin{theorem}
Consider the $n$-body problem with equal masses in ${\bf S}^2$. Then, 
for any $n$ odd, $m>0$ and $z\in(-1,1)$, there are a positive and a negative
$\omega$ that produce elliptic relative equilibria in which the bodies are at the
vertices of an $n$-gon rotating in the plane $z=$ constant. If $n$ is even, this
property is still true if we exclude the case $z=0$. \label{ngonS}
\end{theorem}
\begin{proof}
There are two cases to discuss: (i) $n$ odd and (ii) $n$ even. 

(i) To simplify the presentation, we further denote the bodies by $m_i,
i=-s, -s+1,\dots, -1, 0, 1,\dots, s-1,s$, where $s$ is a positive integer,
and assume that they all have mass $m$.
Without loss of generality we can further substitute into equations (\ref{coordS}) a solution of the form
(\ref{check}) with $i$ as above, $\alpha_{-s}=-{4s\pi\over2s+1},
\dots,\alpha_{-1}=-{2\pi\over2s+1},\alpha_0=0$, $\alpha_1={2\pi\over2s+1},\dots, \alpha_s={4s\pi\over2s+1}$, $r:=r_i, z:=z_i$, and consider only the equations
for $i=0$. The study of this case suffices due to the involved symmetry, which yields 
the same conclusions for any value of $i$.

The equation corresponding to the $z_0$ coordinate takes the form
$$
\sum_{j=-s, j\ne 0}^s{m(z-k_{0j}z)\over(1-k_{0j}^2)^{3/2}}-r^2\omega^2z=0,
$$
where $k_{0j}=x_0x_j+y_0y_j+z_0z_j=\cos\alpha_j-z^2\cos\alpha_j+z^2$. Using
the fact that $r^2+z^2=1$, $\cos\alpha_j=\cos\alpha_{-j}$, and $k_{0j}=k_{0(-j)}$, this equation becomes
\begin{equation}
\sum_{j=1}^s{2(1-\cos\alpha_j)\over(1-k_{0j}^2)^{3/2}}={\omega^2\over m}.
\label{z0}
\end{equation}
Now we need to check whether the equations corresponding to $x_0$ and $y_0$
lead to the same equation. In fact, checking for $x_0$, and ignoring $y_0$, suffices 
due to the same symmetry reasons invoked earlier or the duality of the trigonometric
functions $\sin$ and $\cos$. The substitution of the the above functions into the first equation of (\ref{coordS}) leads us to
$$
(r^2-1)\omega^2\cos\omega t=\sum_{j=-s, j\ne 0}^s{m[\cos(\omega t+\alpha_j)
-k_{0j}\cos\omega t]\over(1-k_{0j}^2)^{3/2}}.
$$
A straightforward computation, which uses the fact that $r^2+z^2=1$, $\sin\alpha_j=-\sin\alpha_{-j}$,  $\cos\alpha_j=\cos\alpha_{-j}$, and $k_{0j}=k_{0(-j)}$, yields the same equation (\ref{z0}). Writing the denominator of
equation (\ref{z0}) explicitly, we are led to
\begin{equation}
\sum_{j=1}^s{2\over(1-\cos\alpha_j)^{1/2}(1-z^2)^{3/2}[2-(1-\cos\alpha_j)(1-z^2)]^{3/2}}
={\omega^2\over m}.
\end{equation}
The left hand side is always positive, so for any $m>0$ and $z\in(-1,1)$ fixed,
there are a positive and a negative $\omega$ that satisfy the equation. Therefore
the $n$-gon with an odd number of sides is an elliptic relative equilibrium.

(ii) To simplify the presentation when $n$ is even, we denote the bodies by $m_i,\ i=-s+1,\dots, -1, 0, 1,\dots, s-1,s$, where $s$ is a positive integer,
and assume that they all have mass $m$. Without loss of generality, we can
substitute into equations (\ref{coordS}) a solution of the form (\ref{check}) with
$i$ as above, $\alpha_{-s+1}={(-s+1)\pi\over s},\dots,\alpha_{-1}=-{\pi\over s}$,
$\alpha_0=0$, $\alpha_1={\pi\over s},\dots,\alpha_{s-1}={(s-1)\pi\over s}$,
$\alpha_s=\pi$, $r:=r_i$, $z:=z_i$, and consider as in the previous case only
the equations for $i=0$. Then using the fact that $k_{0j}=k_{0(-j)}$, $\cos\alpha_j=\cos\alpha_{-j}$, and $\cos\pi=-1$, a straightforward computation brings the equation corresponding to $z_0$ to the form
\begin{equation}
\sum_{j=1}^{s-1}{2(1-\cos\alpha_j)\over(1-k_{0j}^2)^{3/2}}+{2\over(1-k_{0s}^2)^{3/2}}=
{\omega^2\over m}.\label{form}
\end{equation}
Using additionally the relations $\sin\alpha_j=-\sin\alpha_{-j}$ and $\sin\pi=0$, 
we obtain for the equation corresponding to $x_0$ the same form (\ref{form}),
which---for $k_{0j}$ and $k_{0s}$ written explicitly---becomes
\begin{multline*}
\sum_{j=1}^{s-1}{2\over(1-\cos\alpha_j)^{1/2}(1-z^2)^{3/2}[2-(1-\cos\alpha_j)(1-z^2)]^{3/2}}\\
+{1\over 4z^2|z|(1-z^2)^{3/2}}={\omega^2\over m}.
\end{multline*}
Since the left hand side of this equations is positive and finite, given any $m>0$
and $z\in(-1,0)\cup(0,1)$, there are a positive and a negative $\omega$ that satisfy it.
So except for the case $z=0$, which introduces antipodal singularities, the
rotating $n$-gon with an even number of sides is an elliptic relative equilibrium.
\end{proof}

\subsection{Lagrangian solutions} The case $n=3$ presents particular interest in the Euclidean case because the equilateral triangle is an elliptic relative equilibrium for any values of the masses, not only when the masses are equal. But before we check whether this fact holds in ${\bf S}^2$, let us consider the case of three equal masses in more
detail.


\begin{cor}
Consider the 3-body problem with equal masses, $m:=m_1=m_2=m_3$, in ${\bf S}^2$.
Then for any $m>0$ and $z\in(-1,1)$, there are a positive and a negative
$\omega$ that produce elliptic relative equilibria in which the bodies are at the
vertices of an equilateral triangle that rotates in the plane $z=$ constant.
Moreover, for every $\omega^2/m$ there are two values of $z$ that lead
to relative equilibria if $\omega^2/m\in(8/\sqrt{3},\infty)\cup\{3\}$,
three values if $\omega^2/m=8/\sqrt{3}$, and four values if
$\omega^2/m\in(3,8/\sqrt{3})$.
\label{equilateralS}
\end{cor}
\begin{proof}
The first part of the statement is a consequence of Theorem \ref{ngonS} for
$n=3$. Alternatively, we can substitute into system (\ref{coordS}) a 
solution of the form (\ref{check}) with $i=1,2,3$, $r:=r_1=r_2=r_3$, $z=\pm(1-r^2)^{1/2}$, $\alpha_1=0,
\alpha_2=2\pi/3, \alpha_3=4\pi/3$, and obtain the equation
\begin{equation}
{8\over\sqrt{3}(1+2z^2-3z^4)^{3/2}}={\omega^2\over m}.\label{eq4}
\end{equation}
The left hand side is positive for $z\in(-1,1)$ and tends to infinity when 
$z\to\pm1$ (see Figure \ref{3gon}). So for any $z$ in this interval and $m>0$, 
there are a positive and a negative $\omega$ for which the above equation is
satisfied. Figure \ref{3gon} and a straightforward computation also clarify the 
second part of the statement.
\end{proof}

\begin{figure} 
\centering 
\includegraphics[width=1.7in]{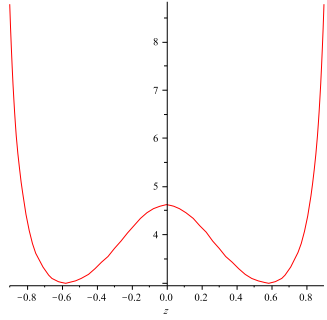}
\caption{\small The graph of the function $f(z)={8\over\sqrt{3}(1+2z^2-3z^4)^{3/2}}$ for
$z\in(-1,1)$.}\label{3gon}
\end{figure}

\begin{remark}
A result similar to Corollary \ref{equilateralS} can be proved for two equal
masses that rotate on a non-geodesic circle, when the bodies are 
situated at opposite ends of a rotating diameter. Then, for $z\in(-1,0)\cup(0,1)$, the 
analogue of (\ref{eq4}) is the equation
$${1\over 4z^2|z|(1-z^2)^{3/2}}={\omega^2\over m}.$$
The case $z=0$ yields no solution because it involves an antipodal singularity.
\end{remark}

We have reached now the point when we can decide whether the equilateral
triangle can be an elliptic relative equilibrium in ${\bf S}^2$ if the masses are not equal. The following result shows that, unlike in the Euclidean case, the answer is negative when the bodies move on the sphere in the same Euclidean plane.


\begin{proposition}
In the 3-body problem in ${\bf S}^2$, if the bodies $m_1, m_2, m_3$
are initially at the vertices of an equilateral triangle in the plane $z=$
constant for some $z\in(-1,1)$, then there are initial velocities that lead
to an elliptic relative equilibrium in which the triangle rotates in its own plane if 
and only if $m_1=m_2=m_3$.\label{equil}
\end{proposition}
\begin{proof}
The implication which shows that if $m_1=m_2=m_3$, the rotating 
equilateral triangle is a relative equilibrium, follows from Theorem 
\ref{equilateralS}. To prove the other implication, we substitute into 
equations (\ref{coordS}) a solution of
the form (\ref{check}) with $i=1,2,3,\ r:=r_1,r_2,r_3,\ z:=z_1=z_2=z_3=\pm
(1-r^2)^{1/2}$, and $\alpha_1=0, \alpha_2=2\pi/3, \alpha_3=4\pi/3$. The
computations then lead to the system
\begin{equation}
\begin{cases}
m_1+m_2=\gamma\omega^2\cr
m_2+m_3=\gamma\omega^2\cr
m_3+m_1=\gamma\omega^2,\cr
\end{cases}
\end{equation}
where $\gamma=\sqrt{3}(1+2z^2-3z^4)^{3/2}/4$. But for any $z=$ constant
in the interval $(-1,1)$, the above system has a solution only for 
$m_1=m_2=m_3=\gamma\omega^2/2$. Therefore the masses must be equal.
\end{proof}

The next result leads to the conclusion that Lagrangian solutions in ${\bf S}^2$ can take place only in Euclidean planes of ${\bf R}^3$. This property is known to be true in the Euclidean case for all elliptic relative equilibria, \cite{Win}, but Wintner's proof doesn't work in our case because it uses the integral of the center of mass. Most
importantly, our result also implies that Lagrangian orbits with non-equal masses cannot exist in ${\bf S}^2$.


\begin{theorem}
For all Lagrangian solutions in ${\bf S}^2$, the masses $m_1, m_2$ and $m_3$ have to rotate on the same circle, whose plane must be orthogonal to the rotation axis, and
therefore $m_1=m_2=m_3$.\label{lagranS}
\end{theorem}
\begin{proof}
Consider a Lagrangian solution in ${\bf S}^2$ with bodies of masses $m_1, m_2$, and $m_3$. This means that the solution, which is an elliptic relative equilibrium, must
have the form
\begin{align*}
x_1&=r_1\cos\omega t,& y_1&=r_1\sin\omega t,& z_1&=(1-r_1^2)^{1/2},\\
x_2&=r_2\cos(\omega t+a),& y_2&=r_2\sin(\omega t+a),& z_2&=(1-r_2^2)^{1/2},\\
x_3&=r_3\cos(\omega t+b),& y_3&=r_3\sin(\omega t+b),& z_3&=(1-r_3^2)^{1/2},
\end{align*}
with $b>a>0$. In other words, we assume that this equilateral forms a constant angle with the rotation axis, $z$, such that each body describes its own circle on ${\bf S}^2$.
But for such a solution to exist, it is necessary that the total angular momentum is
either zero or is given by a vector parallel with the $z$ axis. Otherwise this vector rotates around the $z$ axis, in violation of the angular-momentum integrals. This means that at least the first two components of the vector $\sum_{i=1}^3m_i{\bf q}_i\times\dot{\bf q}_i$ must be zero. A straightforward computation shows this constraint to lead to the condition 
$$
m_1r_1z_1\sin\omega t+m_2r_2z_2\sin(\omega t+a)+m_3r_3z_3\sin(\omega t+b)=0,
$$
assuming that $\omega\ne 0$. For $t=0$, this equation becomes
\begin{equation}
m_2r_2z_2\sin a=-m_3r_3z_3\sin b.
\label{withz}
\end{equation}
Using now the fact that
$$
\alpha:=x_1x_2+y_1y_2+z_1z_2=x_1x_3+y_1y_3+z_1z_3=x_3x_2+y_3y_2+z_3z_2
$$
is constant because the triangle is equilateral, the equation of the system 
of motion corresponding to $\ddot{y}_1$ takes the form
$$
Kr_1(r_1^2-1)\omega^2\sin\omega t=
m_2r_2\sin(\omega t+a)+m_3r_3\sin(\omega t+b),
$$
where $K$ is a nonzero constant. For $t=0$, this equation becomes
\begin{equation}
m_2r_2\sin a=-m_3r_3\sin b.
\label{withoutz}
\end{equation}
Dividing \eqref{withz} by \eqref{withoutz}, we obtain that $z_2=z_3$. Similarly, 
we can show that $z_1=z_2=z_3$, therefore the motion
must take place in the same Euclidian plane on a circle orthogonal to the
rotation axis. Proposition \ref{equil} then implies that $m_1=m_2=m_3$.
\end{proof}


\subsection{Eulerian solutions} 
It is now natural to ask whether such elliptic relative equilibria exist, since---as  Theorem \ref{rengon} shows---they cannot be generated from regular $n$-gons. The answer in the case $n=3$ of equal masses is given by the following result.


\begin{theorem}
Consider the 3-body problem in ${\bf S}^2$ with equal masses, $m:=m_1=m_2=m_3$.
Fix the body of mass $m_1$ at $(0,0,1)$ and the bodies of masses $m_2$ and $m_3$
at the opposite ends of a diameter on the circle $z=$ constant. Then, for any $m>0$ and $z\in(-0.5,0)\cup(0,1)$, there are a positive and a negative $\omega$ that produce 
elliptic relative equilibria.\label{regeo3}
\end{theorem}
\begin{proof}
Substituting into the equations of motion (\ref{coordS}) a solution of the form  
$$x_1=0,\ y_1=0,\ z_1=1,$$
$$x_2=r\cos\omega t,\ y_2=r\sin\omega t, z_2=z,$$
$$x_3=r\cos(\omega t+\pi),\ y_3=r\sin(\omega t+\pi), z_3=z,$$
with $r\ge0$ and $z$ constants satisfying $r^2+z^2=1$, leads either to identities 
or to the algebraic equation
\begin{equation}
{4z+|z|^{-1}\over 4z^2(1-z^2)^{3/2}}={\omega^2\over m}.\label{ratio1}
\end{equation}
The function on the left hand side is negative for $z\in(-1,-0.5)$,
$0$ at $z=-0.5$, positive for $z\in(-0.5,0)\cup(0,1)$, and undefined at $z=0$.
Therefore, for every $m>0$ and $z\in(-0.5,0)\cup(0,1)$, there are a positive and 
a negative $\omega$ that lead to a geodesic relative equilibrium. For $z=-0.5$, we
recover the equilateral fixed point. The sign of $\omega$ determines the sense of 
rotation.
\end{proof}
\begin{figure} 
\centering 
\includegraphics[width=1.7in]{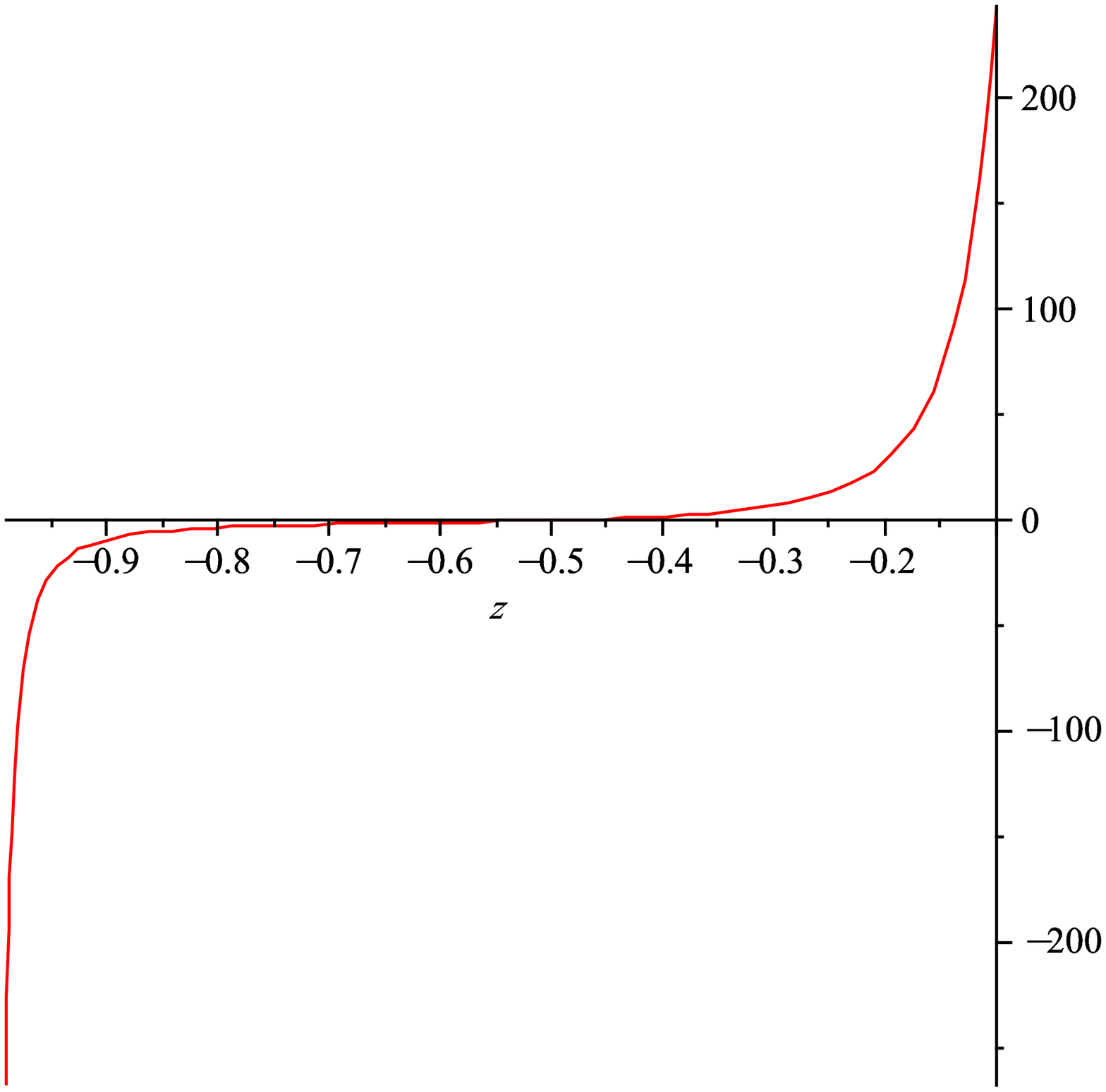}
\includegraphics[width=1.7in]{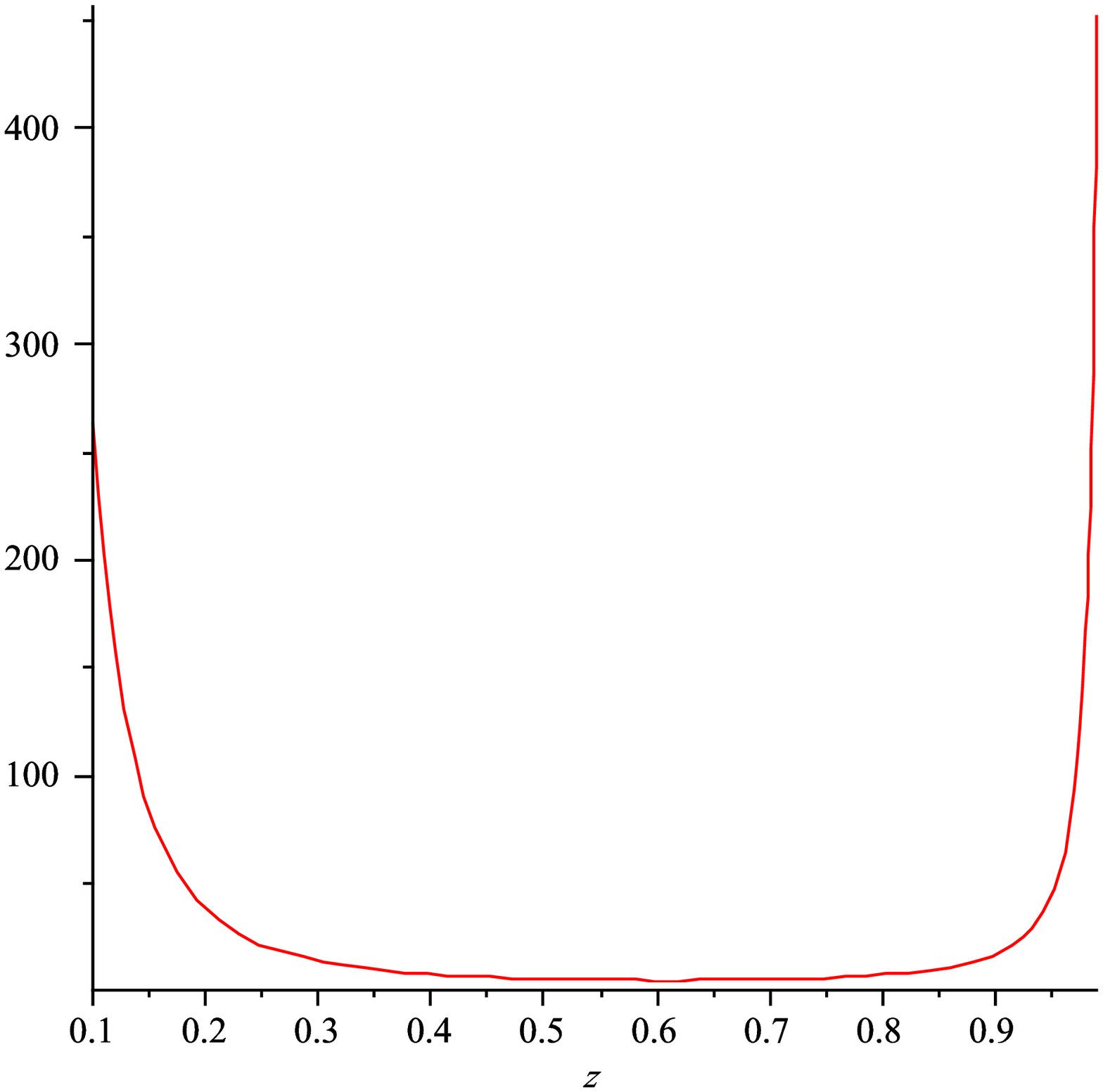}
\caption{\small The graph of the function $f(z)={4z+|z|^{-1}\over 4z^2(1-z^2)^{3/2}}$ in
the intervals $(-1,0)$ and $(0,1)$, respectively.}\label{regeoS2}
\end{figure}
\begin{remark} For every $\omega^2/m\in(0,64\sqrt{15}/45)$, there are three values of 
$z$ that satisfy relation (\ref{ratio1}): one in the interval $(-0.5,0)$ and two 
in the interval $(0,1)$ (see Figure \ref{regeoS2}).
\end{remark}
\begin{remark}
If in Theorem \ref{regeo3} we take the masses $m_1=:m$ and $m_2=m_3=:M$,
the analogue of equation \eqref{ratio1} is
\begin{equation*}
{4mz+M|z|^{-1}\over 4z^2(1-z^2)^{3/2}}={\omega^2\over m}.\label{ratio2}
\end{equation*}
Then solutions exist for any $z\in(-\sqrt{M/m}/2,0)\cup(0,1)$. This means
that there are no fixed points for $M\ge 4m$ (a fact that agrees with what
we learned from Remark \ref{rem} and the proof of Theorem \ref{singularity}), so relative equilibria
exist for such masses for all $z\in(-1,0)\cup(0,1)$.
\end{remark}


\section{Relative equilibria in ${\bf H}^2$}

In this section we will prove  a few results about fixed points, as well as elliptic
and hyperbolic relative equilibria in ${\bf H}^2$. We also show that parabolic relative equilibria do not exist. Since, by the Principal Axis theorem for the Lorentz group, every Lorentzian rotation (see Appendix) can be written, in some basis, either as an elliptic rotation about the $z$ axis, or as an hyperbolic rotation about the $x$ axis, or as a parabolic rotation about the line $x=0$, $y=z$, we can define three kinds of relative equilibria: the elliptic relative equilibria, the hyperbolic relative equilibria, and the parabolic relative equilibria. This terminology matches the standard terminology of hyperbolic geometry \cite{Henle}.

\medskip

The  elliptic relative equilibria  are defined as follows.

\begin{definition}
An elliptic relative equilibrium in ${\bf H}^2$ is 
a solution ${\bf q}_i=(x_i,y_i,z_i)$, $i=1,\dots,n,$ of equations (\ref{coordH}) 
with $x_i=\rho_i\cos(\omega t+\alpha_i), y_i=\rho_i\sin(\omega t+\alpha_i)$, 
and $z_i=(\rho_i^2+1)^{1/2}$,  where $\omega, \alpha_i,$ and 
$\rho_i,\ i=1,\dots,n,$ are constants.
\label{reH1}
\end{definition}

Remark that, as in ${\bf S}^2$, a ``weak'' property of the center of
mass occurs in ${\bf H}^2$ for elliptic relative equilibria. Indeed, if all the bodies 
are at all times on one side of a plane containing the rotation axis, then the integrals of the angular momentum are violated because the vector representing the total angular momentum cannot be zero or parallel to the $z$ axis.

\medskip

Let us now define the hyperbolic relative equilibria.  
\begin{definition}
A hyperbolic relative equilibrium in ${\bf H}^2$ is 
a solution of equations (\ref{coordH}) of the form ${\bf q}_i=(x_i,y_i,z_i),\ i=1,\dots,n,$ defined for all $t\in{\bf R}$, with 
\begin{equation}
x_i={\rm constant},\ \ y_i=\rho_i\sinh(\omega t+\alpha_i),\ \ {\rm and} \ \ 
z_i=\rho_i\cosh(\omega t+\alpha_i),\label{hyper}
\end{equation}
where $\omega, \alpha_i,$ and $\rho_i=(1+x_i^2)^{1/2}\ge1,\ i=1,\dots,n,$ are constants.
\label{reH}
\end{definition}

Finally, the parabolic relative equilibria are defined as follows.
\begin{definition}
A parabolic relative equilibrium in ${\bf H}^2$ is 
a solution of equations (\ref{coordH}) of the form ${\bf q}_i=(x_i,y_i,z_i),\ i=1,\dots,n,$ defined for all $t\in{\bf R}$, with 
\begin{equation}
\begin{split}
x_i&=a_i-b_it+c_it\\
y_i&=a_it+b_i(1-t^2/2)+c_it^2/2\\
z_i&=a_it-b_it^2/2+c_i(1+t^2/2),\label{parabol}
\end{split}
\end{equation}
where $a_i,b_i$ and $c_i,\ i=1,\dots,n,$ are constants, and $a_i^2+b_i^2-c_i^2=-1$.
\label{reH}
\end{definition}


\subsection{Fixed Points in $\bf H^2$}
The simplest solutions of the equations of motion are the fixed points. They
can be seen as trivial elliptic relative equilibria that correspond to $\omega=0$. In
terms of the equations of motion, we can define them as follows.

\begin{definition}
A solution of system (\ref{HamH}) is called a fixed point if 
$$\overline\nabla_{{\bf q}_i}U_{-1}({\bf q})(t)={\bf p}_i(t)={\bf 0}\ \ {\rm for} \ {\rm all}\ \
t\in{\bf R}\ \ {\rm and}\ \ i=1,\dots,n.$$
\end{definition}

Unlike in ${\bf S}^2$, there are no fixed points in ${\bf H}^2$. Let us formally
state and prove this fact.


\begin{theorem}
In the $n$-body problem with $n\ge 2$ in ${\bf H}^2$ there are no configurations that 
correspond to fixed points of the equations of motion.\label{nofixH}
\end{theorem} 
\begin{proof}
Consider any collisionless configuration of $n$ bodies initially at rest in 
${\bf H}^2$. This means that the component of the forces acting on bodies
due to the constraints, which involve the factors $\dot{x}_i^2+\dot{y}_i^2-\dot{z}_i^2,
\ i=1,\dots,n$, are zero at $t=0$. At least one body, $m_i$, has 
the largest $z$ coordinate.  
Notice that the interaction between $m_i$ and any other body takes place 
along geodesics, which are concave-up hyperbolas on the ($z>0$)-sheet of
the hyperboloid modeling
${\bf H}^2$. Then the body $m_j,  j\ne i$, exercises an attraction
on $m_i$ down the geodesic hyperbola that connects these bodies, so the $z$ coordinate of this force acting on $m_i$ 
is negative, independently of whether $z_j(0)<z_i(0)$ or $z_j(0)=z_i(0)$. 
Since this is true for every $j=1,\dots,n,\ j\ne i$, it follows that $\ddot{z}_i(0)<0$. 
Therefore $m_i$ moves downwards the hyperboloid, so the original configuration
is not a fixed point.
\end{proof}

\subsection{Elliptic Relative Equilibria in $\bf H^2$}


We now consider elliptic relative equilibria, and prove an analogue of Theorem \ref{ngonS}.

\begin{theorem}
Consider the $n$-body problem with equal masses in ${\bf H}^2$. Then, 
for any $m>0$ and $z>1$, there are a positive and a negative
$\omega$ that produce elliptic relative equilibria in which the bodies are at the
vertices of an $n$-gon rotating in the plane $z=$ constant.  \label{ngonH}
\end{theorem}

\begin{proof}
The proof works in the same way as for Theorem \ref{ngonS}, by considering
the cases $n$ odd and even separately. The only differences are that we
replace $r$ with $\rho$, the relation $r^2+z^2=1$ with $z^2=\rho^2+1$, 
and the denominator $(1-k_{0j}^2)^{3/2}$ with $(c_{0j}^2-1)^{3/2}$, wherever
it appears, where $c_{0j}=-k_{0j}$ replaces $k_{0j}$. Unlike in 
${\bf S}^2$, the case $n$ even is satisfied for all admissible values of $z$.
\end{proof}

Like in ${\bf S}^2$, the equilateral triangle presents particular interest,
so let us say a bit more about it than in the general case of the regular $n$-gon.


\begin{cor}
Consider the 3-body  with equal masses, $m:=m_1=m_2=m_3$, in ${\bf H}^2$.
Then for any $m>0$ and $z>1$, there are a positive and a negative
$\omega$ that produce relative elliptic equilibria in which the bodies are at the
vertices of an equilateral triangle that rotates in the plane $z=$ constant.
Moreover, for every $\omega^2/m>0$ there is a unique $z>1$ as above.
\label{equilateralH}
\end{cor}
\begin{proof}
Substituting in system (\ref{coordH}) a solution of the form 
\begin{equation}
x_i=\rho\cos(\omega t+\alpha_i),\ \ y_i=\rho\sin(\omega t+\alpha_i),\ \ z_i=z,
\label{subst}
\end{equation}
with $z=\sqrt{\rho^2+1}$, $\alpha_1=0,
\alpha_2=2\pi/3, \alpha_3=4\pi/3$, we are led
to the equation
\begin{equation}
{8\over\sqrt{3}(3z^4-2z^2-1)^{3/2}}={\omega^2\over m}.\label{eq5}
\end{equation}
The left hand side is positive for $z>1$, tends to infinity when 
$z\to1$, and tends to zero when $z\to\infty$. So for any $z$ in this interval and $m>0$, there are a positive and a negative $\omega$ for which the above equation is
satisfied.
\end{proof}

As we already proved in the previous section, an equilateral triangle rotating in its own plane forms an elliptic relative equilibrium in ${\bf S}^2$ only if the three masses lying at its vertices are equal. The same result is true in ${\bf H}^2$, as we will further show.


\begin{proposition}
In the 3-body problem in ${\bf H}^2$, if the bodies $m_1, m_2, m_3$
are initially at the vertices of an equilateral triangle in the plane $z=$
constant for some $z>1$, then there are initial velocities that lead
to an elliptic relative equilibrium in which the triangle rotates in its own plane if 
and only if $m_1=m_2=m_3$.\label{equilH}
\end{proposition}
\begin{proof}
The implication which shows that if $m_1=m_2=m_3$, the rotating 
equilateral triangle is an elliptic relative equilibrium, follows from Theorem \ref{ngonH}. 
To prove the other implication, we substitute into equations (\ref{coordH}) a solution 
of the form (\ref{subst}) with $i=1,2,3,\ \rho:=\rho_1, \rho_2, \rho_3,\ z:=z_1=z_2=z_3=
(\rho^2+1)^{1/2}$, and $\alpha_1=0, \alpha_2=2\pi/3, \alpha_3=4\pi/3$. The
computations then lead to the system
\begin{equation}
\begin{cases}
m_1+m_2=\zeta\omega^2\cr
m_2+m_3=\zeta\omega^2\cr
m_3+m_1=\zeta\omega^2,\cr
\end{cases}
\end{equation}
where $\zeta=\sqrt{3}(3z^4-2z^2-1)^{3/2}/4$. But for any $z=$ constant
with $z>1$, the above system has a solution only for 
$m_1=m_2=m_3=\zeta\omega^2/2$. Therefore the masses must be equal.
\end{proof}

The following result perfectly resembles Theorem \ref{lagranS}. The proof
works the same way, by just replacing the elliptical trigonometric 
functions with hyperbolic ones and changing the signs to reflect the
equations of motion in ${\bf H}^2$.

\begin{theorem}
For all Lagrangian solutions in ${\bf H}^2$, the masses $m_1, m_2$ and $m_3$ have to rotate on the same circle, whose plane must be orthogonal to the rotation axis, and
therefore $m_1=m_2=m_3$.\label{lagranH}
\end{theorem}

We will further prove an analogue of Theorem \ref{regeo3}.


\begin{theorem}
Consider the 3-body problem in ${\bf H}^2$ with equal masses, $m:=m_1=m_2=m_3$.
Fix the body of mass $m_1$ at $(0,0,1)$ and the bodies of masses $m_2$ and $m_3$
at the opposite ends of a diameter on the circle $z=$ constant. Then, for any $m>0$ and $z>1$, there are a positive and a negative $\omega$, which produce elliptic relative equilibria that rotate around the $z$ axis.\label{regeo3H}
\end{theorem}
\begin{proof}
Substituting into the equations of motion (\ref{coordH}) a solution of the form
\begin{align*}
x_1&=0,& y_1&=0,& z_1&=1,\\
x_2&=\rho\cos\omega t,& y_2&=\rho\sin\omega t,& z_2&=z,\\
x_3&=\rho\cos(\omega t+\pi),& y_3&=\rho\sin(\omega t+\pi),& z_3&=z,
\end{align*}
where $\rho\ge0$ and $z\ge1$ are constants satisfying $z^2=\rho^2+1$, leads either to identities or to the algebraic equation
\begin{equation}
{4z^2+1\over 4z^3(z^2-1)^{3/2}}={\omega^2\over m}.\label{ratio2}
\end{equation}
The function on the left hand side is positive for $z>1$.
Therefore, for every $m>0$ and $z>1$, there are a positive and 
a negative $\omega$ that lead to a geodesic elliptic relative equilibrium. The sign of $\omega$ determines the sense of rotation.
\end{proof}
\begin{remark}
For every $\omega^2/m>0$, there is exactly one $z>1$ that satisfies
equation (\ref{ratio2}) (see Figure \ref{regeoH2}).
\end{remark}

\begin{figure} 
\centering 
\includegraphics[width=1.7in]{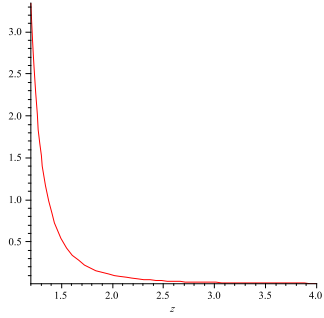}
\caption{\small The graph of the function $f(z)={4z^2+1\over 4z^3(z^2-1)^{3/2}}$ for
$z>1$.}\label{regeoH2}
\end{figure}


\subsection{Hyperbolic Relative Equilibria in $\bf H^2$}
We now prove some results concerning hyperbolic relative equilibria.
We first show that, in the $n$-body problem, hyperbolic relative equilibria 
do not exist along any given fixed geodesic of ${\bf H}^2$. In other words, the
bodies cannot chase each other along a geodesic and maintain the same
initial distances for all times.


\begin{theorem}
Along any fixed geodesic, the $n$-body problem in ${\bf H}^2$ has no
hyperbolic relative equilibria.\label{noreH}
\end{theorem}
\begin{proof}
Without loss of generality, we can prove this result for the geodesic $x=0$.
We will show that equations (\ref{coordH}) do not have solutions of the 
form (\ref{hyper}) with $x_i=0$ and (consequently) $\rho_i=1,\ i=1,\dots,n$.
Substituting
\begin{equation}
x_i=0,\ \ y_i=\sinh(\omega t+\alpha_i),\ \ {\rm and} \ \ 
z_i=\cosh(\omega t+\alpha_i)\label{hypre}
\end{equation}
into system (\ref{coordH}), the equation corresponding to the $y_i$ coordinate becomes
\begin{equation}
\sum_{j=1,j\ne i}^n{m_j[\sinh(\omega t+\alpha_j)-\cosh(\alpha_i-\alpha_j)\sinh(\omega t+\alpha_i)]\over|\sinh(\alpha_i-\alpha_j)|^3}=0.
\label{tri}
\end{equation}
Assume now that $\alpha_i>\alpha_j$ for all $j\ne i$. Let $\alpha_{M(i)}$ be the
maximum of all $\alpha_j$ with $j\ne i$. Then for  
$t\in(-\alpha_{M(i)}/\omega,-\alpha_i/\omega)$, we have that $\sinh(\alpha t+\alpha_j)<0$ for all $j\ne i$ and $\sinh(\alpha t+\alpha_i)>0$. Therefore the left hand side of 
equation (\ref{tri}) is negative in this interval, so the identity cannot take place
for all $t\in{\bf R}$. It follows that a necessary condition to satisfy equation  (\ref{tri}) is that $\alpha_{M(i)}\ge\alpha_i$. But this inequality must be
verified for all $i=1,\dots,n$, a fact that can be written as:
$$\alpha_1\ge\alpha_2 \ \ {\rm or}\ \ \alpha_1\ge\alpha_3\ \ {\rm or}\ \ \dots \ \ {\rm or} \ \
\alpha_1\ge\alpha_n,$$
$$\alpha_2\ge\alpha_1 \ \ {\rm or}\ \ \alpha_2\ge\alpha_3\ \ {\rm or}\ \ \dots \ \ {\rm or} \ \
\alpha_2\ge\alpha_n,$$
$$ \dots $$
$$\alpha_n\ge\alpha_1 \ \ {\rm or}\ \ \alpha_n\ge\alpha_2\ \ {\rm or}\ \ \dots \ \ {\rm or} \ \
\alpha_n\ge\alpha_{n-1}.$$
The constants $\alpha_1,\dots,\alpha_n$ must satisfy one inequality from each 
of the above lines. But every possible choice implies the existence of at least 
one $i$ and one $j$ with $i\ne j$ and $\alpha_i=\alpha_j$. For those $i$ and $j$, 
$\sinh(\alpha_i-\alpha_j)=0$, so equation (\ref{tri}) is undefined, therefore
equations (\ref{coordH}) cannot have solutions of the form (\ref{hypre}). Consequently
hyperbolic relative equilibria do not exist along the geodesic $x=0$.
\end{proof}

Theorem \ref{noreH} raises the question whether hyperbolic relative equilibria do exist at all. For three equal masses, the answer is given by the following result, which shows
that, in ${\bf H}^2$, three bodies can move along hyperbolas lying in parallel planes
of ${\bf R}^3$, maintaining the initial distances among themselves and remaining on the same geodesic (which rotates hyperbolically). The existence of such solutions is surprising. They rather resemble fighter planes flying in formation than celestial bodies moving under the action of gravity alone.


\begin{theorem}
In the 3-body problem of equal masses, $m:=m_1=m_2=m_3$, in ${\bf H}^2$, 
for any given $m>0$ and $x\ne0$, there exist a positive and a negative $\omega$
that lead to hyperbolic relative equilibria. \label{hyp}
\end{theorem}
\begin{proof}
We will show that ${\bf q}_i(t)=(x_i(t),y_i(t),z_i(t)), \ i=1,2,3$, is a hyperbolic relative equilibrium of system (\ref{coordH}) for
\begin{align*}
x_1&=0,& y_1&=\sinh\omega t,&  z_1&=\cosh\omega t,\\
x_2&=x,&  y_2&=\rho\sinh\omega t,&  z_2&=\rho\cosh\omega t,\\
x_3&=-x,&  y_3&=\rho\sinh\omega t,&   z_3&=\rho\cosh\omega t,
\end{align*}
where $\rho=(1+x^2)^{1/2}$.\label{hypre3H}
Notice first that
$$x_1x_2+y_1y_2-z_1z_2=x_1x_3+y_1y_3-z_1z_3=-\rho,$$
$$x_2x_3+y_2y_3-z_2z_3=-2x^2-1,$$
$$
\dot{x}_1^2+\dot{y}_1^2-\dot{z}_1^2=\omega^2,\ \
\dot{x}_2^2+\dot{y}_2^2-\dot{z}_2^2=\dot{x}_3^2+\dot{y}_3^2-\dot{z}_3^2=
\rho^2\omega^2.
$$
Substituting the above coordinates and expressions into equations
(\ref{coordH}), we are led either to identities or to the equation
\begin{equation}
{4x^2+5\over 4x^2|x|(x^2+1)^{3/2}}={\omega^2\over m},\label{eq7}
\end{equation}
from which the statement of the theorem follows.
\end{proof}
\begin{remark}
The left hand side of equation (\ref{eq7}) is undefined for $x=0$, but it tends
to infinity when $x\to0$ and to 0 when $x\to\pm\infty$. This means 
that for each $\omega^2/m>0$ there are exactly one positive and
one negative $x$ (equal in absolute value), which satisfy the equation.
\end{remark}
\begin{remark}
Theorem \ref{hyp} is also true if, say, $m:=m_1$ and $M:=m_2=m_3$.
Then the analogue of equation (\ref{eq7}) is
$$
{m\over x^2|x|(x^2+1)^{1/2}}+{M\over4x^2|x|(x^2+1)^{3/2}}=\omega^2,
$$
and it is obvious that for any $m,M>0$ and $x\ne 0$, there are a positive
and negative $\omega$ satisfying the above equation.
\end{remark}
\begin{remark}
Theorem \ref{hyp} also works for two bodies of equal masses, $m:=m_1=m_2$,
of coordinates
$$
x_1=-x_2=x, y_1=y_2=\rho\sinh\omega t, z_1=z_2=\rho\cosh\omega t,
$$
where $x$ is a positive constant and $\rho=(x^2+1)^{3/2}$. Then the analogue
of equation (\ref{eq7}) is
$$
{1\over 4x^2|x|(x^2+1)^{3/2}}={\omega^2\over m},
$$
which obviously supports a statement similar to the one in Theorem \ref{hyp}.
\end{remark}


\subsection{Parabolic Relative Equilibria in $\bf H^2$}

We now show that there are no parabolic relative equilibria. More precisely, we prove the following result.

\begin{theorem}\label{thpar}
 The $n$-body problem in ${\bf H}^2$ has no parabolic relative equilibria.
\end{theorem}
\begin{proof}
 Let $x_i,y_i$, and $z_i$ be as in the definition of parabolic relative equilibria (\ref{parabol}).
 Then $\dot x_i=-b_i+c_i$, $\dot y_i=a_i+(c_i-b_i)t$, and $a_i+(c_i-b_i)t$. Thus, the first component of the angular momentum is $\sum_i m_i a_i(b_i-c_i) -\sum_i m_i(b_i-c_i)^2t$. It follows that $b_i=c_i$ because the first component of the angular momentum must be constant. 
But $a_i^2+b_i^2-c_i^2=-1$, hence $a_i^2=-1$, which is impossible, since $a_i$  is a real number. 
\end{proof}


\section{ Saari's conjecture}

In 1970, Don Saari conjectured that solutions of the classical $n$-body problem with constant moment of inertia are relative equilibria, \cite{Saa}, \cite{Saa1}. The moment of inertia is defined in classical Newtonian celestial mechanics as ${1\over 2}\sum_{i=1}^nm_i{\bf q}_i\cdot{\bf q}_i$, a
function that gives a crude measure of the bodies' distribution in space. But this definition makes little sense in ${\bf S}^2$ and ${\bf H}^2$ because 
${\bf q}_i\odot{\bf q}_i=\pm1$ for every $i=1,\dots,n$. To avoid this problem, we 
adopt the standard point of view used in physics, and define the moment of inertia in ${\bf S}^2$ or ${\bf H}^2$ about the direction of the angular momentum. But while 
fixing an axis in ${\bf S}^2$  does not restrain generality, the symmetry of ${\bf H}^2$ makes us distinguish between two cases. 

Indeed, in ${\bf S}^2$ we can assume that the rotation takes place around 
the $z$ axis, and thus define the moment of inertia as
\begin{equation}
{\bf I}:=\sum_{i=1}^nm_i(x_i^2+y_i^2).\label{momz}
\end{equation}
In ${\bf H}^2$, all possibilities can be reduced via suitable isometric transformations
(see Appendix) to: (i) the symmetry about the $z$ axis, when the moment of inertia takes the same form (\ref{momz}), and (ii) the symmetry about the $x$ axis, which corresponds to hyperbolic rotations, when---in agreement with the definition of the Lorentz product (see Appendix)---we define the moment of inertia as
\begin{equation}
{\bf J}:=\sum_{i=1}^nm_i(y_i^2-z_i^2).\label{momx}
\end{equation}
The case of the parabolic roations will not be considered because there are no parabolic relative equilibria. 

\smallskip

These definitions allow us to formulate the following conjecture:

\medskip

\noindent {\bf Saari's Conjecture in ${\bf S}^2$ and ${\bf H}^2$}. {\it For the 
gravitational $n$-body problem in ${\bf S}^2$ and ${\bf H}^2$, every solution 
that has a constant moment of inertia about the direction of the angular 
momentum is either an elliptic relative equilibrium in ${\bf S}^2$ or ${\bf H}^2$, 
or a hyperbolic relative equilibrium in ${\bf H}^2$.} 

\medskip

By generalizing an idea we used in the Euclidean case, \cite{Dia}, \cite{Dia2}, 
we can now settle this conjecture when the bodies undergo another constraint. More precisely, we will prove the following result.


\begin{theorem}
For the gravitational $n$-body problem in ${\bf S}^2$ and ${\bf H}^2$, every 
solution with constant moment of inertia about the direction of the angular 
momentum for which the bodies remain aligned along a geodesic that rotates
elliptically in ${\bf S}^2$ or ${\bf H}^2$, or hyperbolically in ${\bf H}^2$, is 
either an elliptic relative equilibrium in ${\bf S}^2$ or ${\bf H}^2$, or a hyperbolic 
relative equilibrium in ${\bf H}^2$.\label{Saari}
\end{theorem}
\begin{proof}
Let us first prove the case in which ${\bf I}$ is constant in ${\bf S}^2$ and ${\bf H}^2$,
i.e.~when the geodesic rotates elliptically.
According to the above definition of $\bf I$, we can assume without loss of generality that the geodesic passes through the point $(0,0,1)$ and rotates about the $z$-axis
with angular velocity $\omega(t)\ne 0$. The angular momentum of each body is ${\bf L}_i=m_i{\bf q}_i\otimes \dot{\bf q}_i$, so its derivative with respect to $t$ takes the form
$$\dot{\bf L}_i=m_i\dot{\bf q}_i\otimes\dot{\bf q}_i+m_i{\bf q}_i\otimes\ddot{\bf q}_i=
m_i{\bf q}_i\otimes{\widetilde\nabla_{{\bf q}_i}} U_\kappa({\bf q})-m_i\dot{\bf q}_i^2{\bf q}_i\otimes{\bf q}_i
=m_i{\bf q}_i\otimes{\widetilde\nabla_{{\bf q}_i}} U_\kappa({\bf q}),$$
with $\kappa=1$ in ${\bf S}^2$ and $\kappa=-1$ in ${\bf H}^2$.
Since ${\bf q}_i\odot{\widetilde\nabla_{{\bf q}_i}} U_\kappa({\bf q})=0$, it follows that 
${\widetilde\nabla_{{\bf q}_i}} U_\kappa({\bf q})$ is either zero or orthogonal to ${\bf q}_i$.
(Recall that orthogonality here is meant in terms of the standard inner product because,
both in ${\bf S}^2$ and ${\bf H}^2$, ${\bf q}_i\odot{\widetilde\nabla_{{\bf q}_i}} U_\kappa({\bf q})={\bf q}_i\cdot{\nabla_{{\bf q}_i}} U_\kappa({\bf q})$.)
If ${\widetilde\nabla_{{\bf q}_i}} U_\kappa({\bf q})={\bf 0}$, then $\dot{\bf L}_i={\bf 0}$, so $\dot{L}_i^z=0$.

Assume now that ${\widetilde\nabla_{{\bf q}_i}} U_\kappa({\bf q})$ is orthogonal to 
${\bf q}_i$. Since all the particles are on a geodesic, their corresponding position
vectors are in the same plane, therefore any linear combination of them is in this
plane, so ${\widetilde\nabla_{{\bf q}_i}} U_\kappa({\bf q})$ is in the same plane. 
Thus $\widetilde\nabla_{{\bf q}_i} U_\kappa({\bf q})$ and ${\bf q}_i$ 
are in a plane orthogonal to the $xy$ plane. It follows that $\dot{\bf L}_i$ is parallel
to the $xy$ plane and orthogonal to the $z$ axis. Thus the $z$ component,
$\dot{L}_i^z$, of $\dot{\bf L}_i$ is $0$, the same conclusion we obtained in the
case ${\widetilde\nabla_{{\bf q}_i}} U_\kappa({\bf q})={\bf 0}$. Consequently, $L_i^z=c_i$, 
where $c_i$ is a constant. 

Let us also remark that since the angular momentum and angular velocity vectors are
parallel to the $z$ axis, $L_i^z={\bf I}_i\omega(t)$, where ${\bf I}_i=m_i(x_i^2+y_i^2)$
is the moment of inertia of the body $m_i$ about the $z$-axis. Since the total moment
of inertia, ${\bf I}$, is constant, and $\omega(t)$ is the same for all bodies because
they belong to the same rotating geodesic, it follows that
$\sum_{i=1}^n{\bf I}_i\omega(t)={\bf I}\omega(t)=c,$
where $c$ is a constant. Consequently, $\omega$ is a constant vector.

Moreover, since $L_i^z=c_i$, it follows that ${\bf I}_i\omega(t)=c_i$. Then
every ${\bf I}_i$ is constant, and so is every $z_i,\ i=1,\dots,n$. Hence each
body of mass $m_i$ has a constant $z_i$-coordinate, and all bodies rotate with 
the same constant angular velocity around the $z$-axis, properties that agree
with our definition of an elliptic relative equilibrium.

We now prove the case ${\bf J}=$ constant, i.e.~when the geodesic rotates hyperbolically in ${\bf H}^2$. According
to the definition of $\bf J$, we can assume that the bodies are on a moving geodesic
whose plane contains the $x$ axis for all time and whose vertex slides along
the geodesic hyperbola $x=0$. (This moving geodesic hyperbola can be also
visualized as the intersection between the sheet $z>0$ of the hyperboloid and 
the plane containing the $x$ axis and rotating about it. For an instant, 
this plane also contains the $z$ axis.)

The angular momentum of each body is ${\bf L}_i=m_i{\bf q}_i\boxtimes\dot{\bf q}_i$, so we can show as before that its derivative takes the form $\dot{\bf L}_i=m_i{\bf q}_i\boxtimes{\overline\nabla_{{\bf q}_i}} U_{-1}({\bf q})$. Again, ${\overline\nabla_{{\bf q}_i}} U_{-1}({\bf q})$ is either zero or orthogonal to ${\bf q}_i$. In the former case we can draw the same conclusion as earlier, that 
$\dot{\bf L}_i={\bf 0}$, so in particular $\dot{L}_i^x=0$. In the latter case, ${\bf q}_i$ 
and ${\overline\nabla_{{\bf q}_i}} U_{-1}({\bf q})$ are in the plane of the moving hyperbola, so their cross product, ${\bf q}_i\boxtimes{\overline\nabla_{{\bf q}_i}} U_{-1}({\bf q})$
(which differs from the standard cross product only by its opposite $z$ component), 
is orthogonal to the $x$ axis, and therefore $\dot{L}_i^x=0$. Thus $\dot{L}_i^x=0$ in either case.

From here the proof proceeds as before by replacing $\bf I$ with $\bf J$ and the $z$ axis with the $x$ axis, and noticing that $L_i^x={\bf J}_i\omega(t)$, to
show that every $m_i$ has a constant $x_i$ coordinate. In other words, each body is moving along a (in general non-geodesic) hyperbola given by the intersection of the hyperboloid with a plane orthogonal to the $x$ axis. These facts in combination with the sliding of the moving geodesic hyperbola along the fixed geodesic hyperbola $x=0$ 
are in agreement with our definition of a hyperbolic relative equilibrium.
\end{proof}


\section{Appendix}

\subsection{The Weierstrass model}

Since the Weierstrass model of the hyperbolic (Bolyai-Lobachevsky) plane is little 
known, we will present here its basic properties. This model appeals for at least 
two reasons: (i) it allows an obvious comparison with the sphere, both from the geometric and analytic point of view; (ii) it emphasizes  the differences between 
the Bolyai-Lobachevsky and the Euclidean plane as clearly as the well-known
differences between the Euclidean plane and the sphere. As far as we are 
concerned, this model was the key for obtaining the results we proved for
the $n$-body problem for $\kappa<0$.

The Weierstrass model is constructed on one of the sheets of the hyperboloid 
$x^2+y^2-z^2=-1$ in the 3-dimensional Minkowski space 
$\mathcal{M}^3:=({\bf R}^3, \boxdot)$, in which  ${\bf a}\boxdot{\bf b}=a_xb_x+a_yb_y-a_zb_z$, with ${\bf a}=(a_x,a_y,a_z)$ and ${\bf b}=(b_x,b_y,b_z)$,
represents the Lorentz inner product. We choose the $z>0$ sheet of the hyperboloid, which we identify with the Bolyai-Lobachevsky plane ${\bf H}^2$.

A linear transformation $T\colon\mathcal{M}^3\to\mathcal{M}^3$ is 
orthogonal if $T({\bf a})\boxdot T({\bf a})={\bf a}\boxdot{\bf a}$ for any
${\bf a}\in\mathcal{M}^3$. The set of these transformations, together with
the Lorentz inner product, forms the orthogonal group
$O(\mathcal{M}^3)$, given by matrices of determinant $\pm1$.
Therefore the group $SO(\mathcal{M}^3)$ of orthogonal transformations of 
determinant 1 is a subgroup of $O(\mathcal{M}^3)$. Another subgroup
of $O(\mathcal{M}^3)$ is $G(\mathcal{M}^3)$, which is formed by
the transformations $T$ that leave ${\bf H}^2$ invariant. Furthermore, $G(\mathcal{M}^3)$ has the closed Lorentz subgroup, ${\rm Lor}(\mathcal{M}^3):= G(\mathcal{M}^3) \cap SO(\mathcal{M}^3)$. 
 
An important result is the Principal Axis Theorem for  ${\rm Lor}(\mathcal{M}^3)$, \cite{Dillen}, \cite{Nomizu}. Let us define the Lorentzian rotations about an axis 
as  the 1-parameter subgroups of  ${\rm Lor}(\mathcal{M}^3)$ that leave the axis pointwise fixed. Then the Principal Axis Theorem states that every Lorentzian transformation has one of the forms: 
$$
A=P\begin{bmatrix}
\cos\theta & -\sin\theta & 0 \\ 
\sin\theta & \cos\theta & 0 \\ 
0 & 0 & 1 
\end{bmatrix} P^{-1},
$$
$$A=P\begin{bmatrix}
1 & 0 & 0 \\ 
0 & \cosh s & \sinh s \\ 
0 & \sinh s & \cosh s 
\end{bmatrix}P^{-1}, 
$$  
or 
$$A=P\begin{bmatrix}
1 & -t & t \\ 
t & 1-t^2/2 & t^2/2 \\ 
t & -t^2/2 & 1+t^2/2 
\end{bmatrix}P^{-1}, 
$$  
where $\theta\in[0,2\pi)$, $s, t\in{\bf R}$, and  $P\in {\rm Lor}(\mathcal{M}^3)$. 
These transformations are called elliptic, hyperbolic, and parabolic, respectively. The elliptic transformations are rotations about a {\it timelike} axis---the $z$ axis in our case; hyperbolic rotations are rotations about a {\it spacelike} axis---the $x$ axis in our context; and parabolic transformations are rotations about 
a {\it lightlike} (or null) axis, represented here by the line  $x=0$, $y=z$.
This result resembles Euler's Principal Axis Theorem, which states that any element of $SO(3)$ can be written, in some orthonormal basis, as a rotation about the $z$ axis.

The geodesics of ${\bf H}^2$ are the hyperbolas obtained by intersecting the 
hyperboloid with planes passing through the origin of the coordinate system.
For any two distinct points ${\bf a}$ and ${\bf b}$ of ${\bf H}^2$, there is a unique geodesic that connects them, and the distance between these points is given by 
$d({\bf a},{\bf b})=\cosh^{-1}(-{\bf a}\boxdot{\bf b})$.

In the framework of Weierstrass's model, the parallels' postulate of hyperbolic geometry can be translated as follows. Take a geodesic $\gamma$, i.e.~a hyperbola obtained by intersecting a plane through the origin, $O$, of the coordinate system with the upper sheet, $z>0$, of the hyperboloid. This hyperbola has two asymptotes in its plane: the straight lines $a$ and $b$, intersecting at $O$. Take a point, $P$, on the upper sheet of the hyperboloid but not on the chosen hyperbola. The plane $aP$ produces the geodesic hyperbola $\alpha$, whereas $bP$ produces $\beta$. These two hyperbolas intersect at $P$. Then $\alpha$ and $\gamma$ are parallel geodesics meeting at infinity along $a$, while $\beta$ and $\gamma$ are parallel geodesics meeting at infinity along $b$. All the hyperbolas between $\alpha$ and $\beta$ (also obtained from planes through $O$) are non-secant with $\gamma$.

Like the Euclidean plane, the abstract Bolyai-Lobachevsky plane has no privileged points or geodesics. But the Weierstrass model has some convenient points and 
geodesics, such as the point $(0,0,1)$ and the geodesics passing through it. The elements of ${\rm Lor}(\mathcal{M}^3)$ allow us to move the geodesics of ${\bf H}^2$ to convenient positions, a property we frequently use in this paper to simplify our arguments. Other properties of the Weierstrass model can be found in \cite{Fab} and \cite{Rey}. The Lorentz group is treated in some detail in \cite{Bak}, but the Principal Axis Theorems for the Lorentz group contained in \cite{Bak} and \cite{Rey} fails to include parabolic rotations, and is therefore incomplete.

\subsection{History of the model}

The first researcher who mentioned Karl Weierstrass in connection with the hyperboloidal model of the Bolyai-Lobachevsky plane was Wilhelm Killing. In a paper published in 1880, \cite{Kil1}, he used what he called Weierstrass's coordinates to describe the ``exterior hyperbolic plane'' as an ``ideal region'' of the Bolyai-Lobachevsky plane. In 1885, he added that Weierstrass had introduced these coordinates, in combination with ``numerous applications,''  during a seminar held in 1872,  \cite{Kil2}, pp.~258-259. We found no evidence of any written account of the hyperboloidal model for the Bolyai-Lobachevsky plane prior to the one Killing gave in a paragraph of \cite{Kil2}, p.~260. His remarks might have inspired Richard Faber to name this model after Weierstrass and to dedicate a chapter to it in \cite{Fab}, pp.~247-278.

\bigskip

\noindent{\bf Acknowledgments.} We would like to thank Sergey Bolotin, Alexey Borisov, Eduardo Pi\~{n}a, Don Saari,  Alexey Shchepetilov, and Jeff Xia for discussions and suggestions that helped us improve this paper. We also acknowledge the grant support Florin Diacu and Manuele Santoprete received from NSERC (Canada) and Ernesto P\'erez-Chavela from CONACYT (Mexico).


\end{document}